\documentclass{amsart}
\usepackage{latexsym,url,amssymb,amsmath,hyperref}
\newtheorem{theorem}{Theorem}[section]
\newtheorem{lemma}[theorem]{Lemma}
\newtheorem{proposition}[theorem]{Proposition}
\newtheorem{corollary}[theorem]{Corollary}
\newtheorem{definition}[theorem]{Definition}
\newtheorem{conjecture}[theorem]{Conjecture}
\newtheorem{remark}[theorem]{Remark}

\usepackage{color}

\newcommand{\Irr}{\mathop{\mathrm{Irr}}}
\newcommand{\Alt}{\mathop{\mathrm{Alt}}}
\newcommand{\Aut}{\mathop{\mathrm{Aut}}}
\newcommand{\Out}{\mathop{\mathrm{Out}}}
\newcommand{\Sym}{\mathop{\mathrm{Sym}}}
\newcommand{\GL}{\mathop{\mathrm{GL}}}
\newcommand{\PSL}{\mathop{\mathrm{PSL}}}
\newcommand{\PGL}{\mathop{\mathrm{PGL}}}

\newcommand{\AGL}{\mathop{\mathrm{AGL}}}
\newcommand{\lcm}{\mathop{\mathrm{lcm}}}
\def\cent#1#2{{\bf C}_{{#1}}({#2})}

\newcommand{\Fix}{\mathop{\mathrm{Fix}}}
\newcommand{\fpr}{\mathop{\mathrm{fpr}}}
\newcommand{\soc}{\mathop{\mathrm{soc}}} 
\renewcommand{\wr}{\mathop{\mathrm{wr}}}

\newcommand{\vrD}{\vrule height1.9ex width0.4pt depth0.8ex}
\newcommand{\vrS}{\vrule height1.26ex width0.3pt depth0.56ex}
\newcommand{\vrSS}{\vrule height0.95ex width0.3pt depth0.4ex}
\def\vr{\mathchoice{\vrD}{\vrD}{\vrS}{\vrSS}}

\def\divides{\mathchoice{\mathrel{\vr}}{\mathrel{\vr}}
                        {\mathrel{\,\vr\,}}{\mathrel{\,\vr\,}}}
\def\Om{\Omega}

\begin{document}

\title[Regular Orbits]{Finite primitive permutation groups and regular cycles
  of their elements}

\author[M.~Giudici]{Michael Giudici}
\address{Michael Giudici, Centre for Mathematics of Symmetry and Computation,\newline
School of Mathematics and Statistics,\newline
The University of Western Australia,
 Crawley, WA 6009, Australia}
\email{michael.giudici@uwa.edu.au}

\author[C. E. Praeger]{Cheryl E. Praeger}
\address{Cheryl E. Praeger, Centre for Mathematics of Symmetry and Computation,\newline
School of Mathematics and Statistics,\newline
The University of Western Australia,
 Crawley, WA 6009, Australia\newline
Also affiliated with King Abdulaziz University,
Jeddah, Saudi Arabia} \email{Cheryl.Praeger@uwa.edu.au}

\author[P. Spiga]{Pablo Spiga}
\address{Pablo Spiga, Dipartimento di Matematica e Applicazioni, University of Milano-Biccoca, Via Cozzi 53, 20125 Milano, Italy} \email{pablo.spiga@unimib.it}

\thanks{The paper forms part of the Australian Research Council Discovery Project grant DP130100106 of the second author.}


\begin{abstract}
We conjecture that if $G$ is a finite primitive group and if $g$ is an element of $G$, 
then either the element $g$ has a cycle of length equal to its order, or for some 
$r,m$ and $k$, the group $G\leq \Sym(m)\wr \Sym(r)$, preserving a product structure 
of $r$ direct copies of the natural action of $\Sym(m)$ or $\Alt(m)$  on $k$-sets. 
In this paper we reduce this conjecture to the case that $G$ is an almost simple group 
with socle a classical group.\\
\begin{center}
 {\em Dedicated to the  memory of  \'Akos Seress}
\end{center}
\end{abstract}
\maketitle

\section{Introduction}\label{intro}

Let $G$ be a finite primitive permutation group on a finite set $\Om$ and let 
$H\leq G$. A \emph{regular orbit} of $H$ in $\Om$ is one of size $|H|$. In particular, 
if $H=\langle g\rangle$, then a regular orbit of $H$ is the point set of a $g$-cycle of 
length equal to the order $|g|$ of $g$ in its disjoint cycle representation; we call 
such a cycle a \emph{regular cycle} of $g$ in $\Omega$.  Siemons and Zalesskii~\cite{SZ2} asked 
for conditions under which subgroups of primitive groups could be guaranteed to have
regular orbits. 

The work of Siemons and Zalesskii in \cite{SZ2,SZ1}, and later the work of Emmett and 
Zalesskii in~\cite{EZal}, focuses
on the case where $H$ is cyclic and $G$ is a non-abelian simple group (though the groups satisfying 
${\rm PSL}(n,q)< G\leq {\rm PGL}(n,q)$ are also treated in \cite[Theorem 1.1]{SZ2}). They show that, 
if $G$ is a non-abelian simple group that is not isomorphic to an alternating group,
and if either $G$ is a classical group or $G$ has a $2$-transitive permutation representation, then 
each element of $G$ has a regular cycle. 

This paper arose from a conversation between the third author and Alex Zalesskii and we are grateful to Alex 
for being so enthusiastic in talking about his mathematics. Our work 
makes a significant contribution towards answering the question: \emph{When do all elements
of a finite primitive permutation group $(G, \Omega)$ have at least one regular cycle?} 
The finite alternating group  $G=\Alt(n)$ acting on $\Omega=\{1,\dots,n\}$, with $n\geq7$, 
does not have this property 
since, for example, the permutation  $g=(123)(45)(67)\in\Alt(n)$ has no cycle of length $|g|=6$ in its 
natural action on $n$ points. However, the family of primitive permutation groups $(G,\Om)$
containing elements with no regular cycles does not consist simply of the 
natural representations of the alternating and symmetric groups  $\Alt(n)$ and $\Sym(n)$. For example, 
an element $g\in\Sym(10)$ with cycles of lengths $2, 3, 5$ in  its action on 10 points, 
has order $30$ and has cycle lengths $1, 3, 5, 5, 6, 10, 15$ in its action on the set of $45$ pairs 
of points from $\{1,\dots,10\}$. Thus the primitive action of $\Sym(10)$ on pairs 
is another example.

The natural action of $\Sym(m)$ on the set of $k$-subsets of $\{1,\dots,m\}$ is called its \emph{$k$-set action}.   
It is possible to specify precisely the $k$-set actions of finite symmetric groups 
for which all elements have regular cycles.

\begin{theorem}\label{actiononsets}
Let $k\geq 1$ and let $n_k$ be the sum of the first $k$ prime numbers.
Then, for $m\geq 2k$,  every element of $\Sym(m)$ has a regular cycle in its $k$-set action if and only if $m < n_{k+1}$.
\end{theorem}

Theorem~\ref{actiononsets} will be proved in Section~\ref{alt}. 
We believe moreover that primitive permutation groups containing elements 
without regular cycles have very restrictive structure. 
Our investigations of these groups led us to make the following conjecture, and then
we attempted to prove it. 
We say that a subgroup $G$ of $\Sym(\Om)$ \emph{preserves a product structure on 
$\Omega$} if $\Omega=\Delta^r$ and $G$ is isomorphic to a subgroup  of 
$\Sym(\Delta)\wr \Sym(r)$ in its product action.

\begin{conjecture}\label{conj}
Let $(G,\Om)$ be primitive such that some element has no regular cycle. 
Then there exist integers $k\geq1, r\geq1$ and $m\geq5$ such that
$G$ preserves a product structure on $\Omega=\Delta^r$ with $|\Delta|=\binom{m}{k}$, and $\Alt(m)^r
\vartriangleleft G\leqslant \Sym(m)\wr \Sym(r)$, where $\Sym(m)$ induces its $k$-set action on $\Delta$.
\end{conjecture}

We note that each value of $k$ and $r$, and each $m\geq n_{k+1}$ (as defined in Theorem~\ref{actiononsets}) 
yield examples in Conjecture~\ref{conj}. This is true since, by Theorem~\ref{actiononsets}, $\Sym(m)$ contains 
elements with no regular cycles in their $k$-set action on $\Delta$, and we observe in Lemma~\ref{lem:trivial} 
that each such element is associated with elements of $\Sym(m)\wr \Sym(r)$ with no regular cycles in their
product action on $\Delta^r$.
 
We go a long way towards proving Conjecture~\ref{conj} in this paper, reducing to the case where $G$ is 
an almost simple classical group. The final case of classical groups is settled by an analysis in 
\cite{GS}, as we discuss below.

\begin{theorem}\label{thrm:main}
Let $(G,\Om)$ be a primitive permutation group containing an element with no regular cycle, and suppose that 
$(G,\Om)$ is not one of the groups in Conjecture~$\ref{conj}$. 
Then $G$ preserves a product structure on $\Omega=\Delta^r$, for some $r\geq1$, and 
$T^r\vartriangleleft G\leqslant H\wr \Sym(r)$ for some classical group $H$ with simple socle 
$T$ such that $(H,\Delta)$ is primitive and some element has no regular cycle. 
\end{theorem}

As we mentioned above (see Lemma~\ref{lem:trivial}), each primitive classical group
$(H,\Delta)$ containing an element with no regular cycle, and each positive integer $r$,
leads to `product action examples' in Theorem~\ref{thrm:main}. Moreover classical groups with this 
property are known, but the examples we mention are permutationally isomorphic to 
groups in Conjecture~\ref{conj}: for example, 
the group $H={\rm P}\Gamma{\rm L}(2,4)\cong \Sym(5)$ of degree $5$ contains an element of order $6$
(see Theorem~\ref{actiononsets}). Also in Proposition~\ref{thm2} we identify a second example, 
namely in $H= \Sym(6)$ with socle $\Alt(6)\cong {\rm PSL}(2,9)$, acting on the cosets of
${\rm PGL}(2,5)$, some elements of order 6 have no regular cycle.
The results of \cite{EZal,SZ2,SZ1} deal with the cases where $H$ is a simple classical group. 
The general case of almost simple classical groups will be handled by Simon Guest and the third
author in \cite{GS}.




Clearly Theorem~\ref{thrm:main} applies to any transitive permutation
group $G$ having a system of imprimitivity $\mathcal{B}$ with $G$
acting faithfully and primitively on $\mathcal{B}$. In particular,
Theorem~\ref{thrm:main} immediately applies to the class of finite
quasiprimitive groups.

 The following interesting results arise while proving Theorem \ref{thrm:main}.

\begin{corollary}
Each element of $\GL(d,q)$ has a regular cycle on vectors.
\end{corollary}

\begin{corollary}\label{cor1.5}
Each automorphism of a finite non-abelian simple group $T$ 
has a regular cycle in its natural action on $T$.
\end{corollary}





We make the following brave conjecture.

\begin{conjecture}
There exists an absolute positive constant $c$ such that, if $(G,\Om)$ is a primitive group,  but
not one of the groups in Conjecture~$\ref{conj}$, then for each $g\in G$,
$$
\frac{o_g}{c_g}>c
$$
where $o_g$ is the number of regular cycles of $g$, and $c_g$ is the total number of cycles of 
$g$ (including cycles of length 1).
\end{conjecture}

We prove Theorem~\ref{actiononsets} in Section~\ref{alt}. In order to prove Theorem~\ref{thrm:main}, we
investigate the cycle lengths of elements of finite primitive permutation groups $(G,\Om)$  
according to the O'Nan-Scott type of $G$. 
The O'Nan-Scott Theorem~\cite[Theorem 4.6]{cameron} asserts that if $G$ is a 
primitive group that does not preserve a product structure then $G$ has one of the following three types:
\begin{itemize}
\item[(1)] \textit{Affine type},
\item[(2)] \textit{Diagonal type},
\item[(3)] \textit{Almost simple}.
\end{itemize}

Primitive groups preserving a product structure are dealt with in Theorem \ref{thrm:PA}, 
and we consider those of affine type in Theorem \ref{thrm:affine} and those of diagonal 
type in Theorem \ref{thrm:SD}. We investigate the almost simple primitive groups 
in Section \ref{sec:AStype}, completely dealing with the sporadic groups, 
alternating and symmetric groups, and exceptional groups of Lie type.

\section{Basic lemmas}\label{reduction}


We start our analysis with some elementary (but very useful) lemmas.

\begin{lemma}\label{basic}
Let $G$ be a permutation group on $\Omega$. Assume that for every
$g\in G$, with $|g|$ square-free, $g$ has a regular cycle. 
Then every element of $G$ has a regular cycle.
\end{lemma}

\begin{proof}
Let $g\in G$. We argue by induction on the number of prime divisors of
$|g|$ (counted with multiplicity). If $|g|$ is square-free, then there
is nothing to
prove. Thus assume that $|g|$ is divisible by $p^2$, for some prime
$p$. Now, since the number of prime divisors of $|g^p|=|g|/p$ is less
than the number of prime divisors of $|g|$, we may apply induction and
hence there exists $\omega\in \Omega$ with $|\omega^{\langle g^p\rangle
}|=|g|/p$. Thus $|\omega^{\langle g\rangle}|=|g|$ or $|g|/p$. However, the latter case would imply that $|\omega^{\langle g^p\rangle
}|=|g|/p^2$, a contradiction. Thus the cycle of $g$ containing $\omega$ has length $|g|$.
\end{proof}

The previous lemma allows us to consider only elements of square-free
order. Before proceeding we introduce a definition.

\begin{definition}{\rm Let $G$ be a permutation group on a
    finite set $\Omega$, let $x\in G$, and
$\Fix_\Omega(x)=\{\omega\in\Omega\mid \omega^x=\omega\}$. Then 
  $\fpr_\Omega(x)=|\Fix_\Omega(x)|/|\Omega|$ is called the \emph{fixed-point-ratio} of $x$.}
\end{definition}

The following lemma is part of the folklore and the sketch of a proof can
be found in~\cite[Lemma~$2.5$]{LS}.
\begin{lemma}\label{obvious3}
Let $G$ be a transitive group on $\Omega$, let $H$ be a point
stabilizer and let $g$ be in
$G$.
Then $$\mathrm{fpr}_\Omega(g)=\frac{|g^G\cap H|}{|g^G|}.$$
\end{lemma}

The following lemma is one of our main tools.

\begin{lemma}\label{lemma:apeman}Let $G$ be a transitive permutation
  group on $\Omega$ and let $g$ be in $G$. Then 
  $g$ has a regular cycle if and only if
\begin{equation}\label{Fix}
\bigcup_{\substack{p \divides |g|\\p\,\mathrm{
        prime}}}\mathrm{Fix}_\Omega(g^{|g|/p})\subsetneq\Omega, 
\end{equation}
In particular, if 
$$
\sum_{\substack{p\divides |g|\\p\, \mathrm{
      prime}}} \mathrm{fpr}_\Omega(g^{|g|/p}) <1,$$
then $g$ has a regular cycle.
\end{lemma}
\begin{proof}

Denote by $\Delta$ the set on the left hand side of~\eqref{Fix}. Assume that $g$ has
a cycle of length $|g|$ and let $\alpha$ be a point from this
cycle. Then $\alpha^{g^{|g|/p}}\neq \alpha$ and $\alpha\notin
\Fix_\Omega(g^{|g|/p})$, for every prime divisor $p$
of $|g|$. Thus $\alpha\notin \Delta$.

Conversely, assume that $\Delta\neq \Omega$ and let $\alpha\in
\Omega\setminus \Delta$. If $|\alpha^{\langle g\rangle}|<|g|$, then
$|\alpha^{\langle g\rangle}|$ divides $|g|/p$, for some
prime divisor $p$ of $|g|$. Thus $\alpha^{g^{|g|/p}}=\alpha$ and
$\alpha\in \Fix_\Omega(g^{|g|/p})$, contradicting our choice of $\alpha$.
\end{proof}

\section{Primitive groups preserving a product structure}\label{PA type}

We begin with the following simple observation.

\begin{lemma}
\label{lem:trivial}
Let $H\leqslant\Sym(\Delta)$. If $h\in H$ has no regular cycle in $\Delta$, then 
$(h,1,\ldots,1)\in H\wr \Sym(\ell)$ has no regular cycle in $\Delta^\ell$.
\end{lemma}

Next we prove our main result for this section.

\begin{theorem}\label{thrm:PA}
Let $H\leqslant\Sym(\Delta)$ such that $|\Delta|>1$ and each 
element of $H$ has a regular cycle in $\Delta$. Then each element of $H\wr \Sym(\ell)$ 
has a regular cycle in $\Delta^\ell$.  
\end{theorem}

\begin{proof}
Let  $G=H\wr \Sym(\ell)$  and write $g\in G$ as $g=(h_1,\ldots,h_\ell)\sigma$, with
$\sigma\in \Sym(\ell)$ and with $h_i\in H$, for each $i\in
\{1,\ldots,\ell\}$. We start by considering the case that $\sigma$ is
a cycle of length $\ell$. In particular, replacing $g$ by a suitable
$G$-conjugate if necessary, we may assume that
$\sigma=(1,\ldots,\ell)$. For each $i\in \{2,\ldots,\ell\}$, write
$x_i=h_ih_{i+1}\cdots h_\ell$. Furthermore, write $x_1=1$ and
$n=(x_1,x_2,\ldots,x_\ell)\in H^\ell$.  Now, we have
\begin{eqnarray*}
g^n&=&n^{-1}(h_1,\ldots,h_\ell)\sigma
n=(x_1^{-1}h_1,x_2^{-1}h_2,\ldots,x_\ell^{-1}h_\ell)n^{\sigma^{-1}}\sigma\\
&&(x_1^{-1}h_1x_2,x_2^{-1}h_2x_3,\ldots,x_{\ell}^{-1}h_\ell x_1)\sigma=(h_1\cdots
h_\ell,1,\dots,1)\sigma.
\end{eqnarray*}
In particular, replacing $g$ by $g^n$, we may assume that
$h_2=\cdots=h_\ell=1$.  Write $x=(h_1,1,\ldots,1)$ so that $g=x\sigma$.
 Observe that $$g^\ell=xx^{\sigma^{-1}}\cdots x^{\sigma^{-(\ell-1)}}=(h_1,\ldots,h_1)$$
and so $|g|=|h_1|\ell$.

Now, let $\delta\in \Delta$ with
$|\delta^{\langle h_1\rangle}|=|h_1|$. Suppose first that $h_1=1$, 
that is, $x=1$. Let $\delta'\in \Delta\setminus\{\delta\}$ and take 
the point $\omega=(\delta',\delta,\ldots,\delta)\in \Omega$. Clearly, 
the cycle of $g=x\sigma=\sigma$ containing $\omega$ has length $\ell=
|g|$. 

Suppose then that $h_1\neq 1$ and take the point
$\omega=(\delta,\ldots,\delta)\in \Omega$. Note that $\delta^{h_1}\ne\delta$ by 
the definition of $\delta$. We show that the cycle of
$g$ containing $\omega$ has length $|g|$, from which the theorem follows
for $\sigma$ an $\ell$-cycle. Fix a positive integer $t$
and write $t=q\ell+r$, with $0\leq r<\ell$. Note that $(x_1,x_2,\dots,x_\ell)^{\sigma^{-1}}
=(x_{1\sigma}, x_{2\sigma},\dots,x_{\ell\sigma})$, and hence for $x=(h_1,1,\dots,1)$,
$x^{\sigma^{-1}}=(1,\dots,1,h_1)$, $x^{\sigma^{-2}}=(1,\dots,1,h_1,1)$, etc, so that

\begin{eqnarray*}
g^t&=&(x\sigma)^t=xx^{\sigma^{-1}}x^{\sigma^{-2}}\cdots
x^{\sigma^{-(t-1)}}\sigma^t\\
&=&(xx^{\sigma^{-1}}\cdots
x^{\sigma^{-(\ell-1)}})^q\,y_r\,\sigma^r\\
\end{eqnarray*}
where $y_r=1$ if $r=0$, and for $0<r<\ell$, 
\[
 y_r = xx^{\sigma^{-1}}\cdots
x^{\sigma^{-(r-1)}} = (h_1, \underbrace{1,\ldots,1}_{\ell-r\,\mathrm{times}}, 
\underbrace{h_1,\ldots,h_1}_{r-1\,\mathrm{times}}).
\]
Using this description of $g^t$, we see that if $w^{g^t}=\omega$ then
$$
(\delta,\ldots,\delta)= \begin{cases}
                         (\delta^{h_1^q},\ldots,\delta^{h_1^q}) &\mbox{if } r=0\\
(\delta^{h_1^{q+1}},\underbrace{\delta^{h_1^q},\ldots,\delta^{h_1^q}}_{\ell-r\,\textrm{times}},
\underbrace{\delta^{h_1^{q+1}},\ldots,\delta^{h_1^{q+1}}}_{r-1\,\textrm{times}})^{\sigma^r} &\mbox{if } 0<r<\ell.
                        \end{cases}
$$
In particular, by applying $\sigma^{-r}$ on both sides of this equality, 
we have $\delta=\delta^{h_1^q}$ and also, if $0<r<\ell$, then $\delta^{h_1^{q+1}}=\delta^{h_1^q}$. 
If $r>0$ then these two conditions imply that  $\delta=\delta^{h_1^q}= \delta^{h_1^{q+1}}=\delta^{h_1}$,
which is a contradiction. Thus $r=0$ and the condition  $\delta=\delta^{h_1^q}$
implies that $|h_1|=|\delta^{\langle h_1\rangle}|$ divides $q$.
Therefore $|g|=\ell
|h_1|$ divides $\ell q=t$. So the cycle of $g$ containing $\omega$ has
length $|g|$.

We now consider the case that $\sigma$ has more than one cycle in its
disjoint cycle decomposition.  Write $\sigma=\sigma_1\cdots \sigma_r$
with $\sigma_1,\ldots,\sigma_r$ the disjoint cycles of $\sigma$ of lengths
$\ell_1,\ldots,\ell_r$, respectively. Replacing $g$ by a suitable
$G$-conjugate, we may assume that
$\sigma=(1,\ldots,\ell_1)(\ell_1+1,\ldots,\ell_1+\ell_2)\cdots$. For each
$i\in \{1,\ldots,r\}$, take $n_i\in H^{\ell}$ such that the
$u^{\textrm{th}}$ coordinate of $n_i$ is $h_u$ if $u$ is in the support of
$\sigma_i$ and is $1$ otherwise. Write now $g_i=n_i\sigma_i$. Clearly,
$g=g_1\cdots g_r$ and $g_ig_j=g_jg_i$ for all $i,j$.  It follows that
$$
|g|=\lcm\{|g_i|\mid i\in \{1,\ldots,r\}\}.
$$
Moreover, we may view $g_i$
as an element of $H\wr \Sym(\ell_i)$ and so, by the previous case,
there exists $\omega_i\in \Delta^{\ell_i}$ with $|\omega_i^{\langle
  g_i\rangle}|=|g_i|$. Now it immediately follows that the cycle of
$g$ containing the point $(\omega_1,\ldots,\omega_r)\in \Delta^\ell$
has length $\lcm\{|g_i|\mid i\in \{1,\ldots,r\}\}=|g|$, and 
the proof is complete.
\end{proof}

\section{Primitive groups of Affine type}\label{HA type}

The main result of this section is Theorem~\ref{thrm:affine}, in which we
prove Theorem~\ref{thrm:main} in the special case where $G$ is a
primitive group of affine type.  In the rest of this section let $V$ be 
the $d$-dimensional vector space of row vectors over the field
$\mathbb{F}_q$ with $q$ elements, where $d\geq 1$ and $q$ is a prime power,
and consider $\GL_d(q)=\GL(V)$ acting naturally
on $V$.

\begin{lemma}\label{GL}
Let $g\in \GL_d(q)$. Write 
$$
\mathcal{S}_g:=\{v\in V\mid v \textrm{ lies in a regular cycle of }g\}.$$ Then
$\mathcal{S}_g$ spans the vector space $V$.
\end{lemma}
\begin{proof}
Write $|g|=p_1^{n_1}\cdots p_r^{n_r}$, with $p_1,\ldots,p_r$ distinct
primes and with $n_i>0$, for each $i\in \{1,\ldots,r\}$. Write $W_g:=\langle \mathcal{S}_g\rangle$.

Now the vector space $V$ decomposes as a direct sum of
indecomposable $\mathbb{F}_q\langle g\rangle$-modules, say
$V=W_1\oplus \cdots \oplus W_s$ with $W_i$ an indecomposable
$\mathbb{F}_q\langle g\rangle$-module for each $i\in
\{1,\ldots,s\}$. (We recall that $W$ is said to be indecomposable if
$W$ cannot be written as $W=U\oplus U'$ with $U$ and $U'$ proper
non-trivial submodules.) Let $g_i$ be the linear transformation induced by $g$
on $W_i$. Clearly,
\begin{equation}\label{aff:1}
|g|=\lcm\{|g_i|\mid i\in \{1,\ldots,s\}\}.
\end{equation} 
For each $i$, let
$\mathcal{S}_i=\{w\in W_i\mid |w^{\langle
  g_i\rangle}|=|g_i|\}$.
We argue by induction
on the number $n(g):=\sum_{1}^rn_i+s$.
If $n(g)=1$ then $s=1$ and $\sum_in_i=0$ so that $g=1$ and 
$\mathcal{S}_g=V$ spans $V$. So assume that $n=n(g)>1$ and that the 
result holds for all elements $g'$ with $n(g')<n$.

Suppose first that $n_i>1$, for some $i\in
\{1,\ldots,r\}$. A moment's thought gives that
$\mathcal{S}_{g}=\mathcal{S}_{g^{p_i}}$. In particular, by induction,
we have that $\mathcal{S}_{g^{p_i}}$ spans $V$ and hence so does
$\mathcal{S}_g$. Thus we may assume that $n_i=1$ for each $i\in
\{1,\ldots,r\}$, that is, $g$ has square-free order. This
observation will simplify some of the computations later in the proof.

Next assume that $s=1$, that is, $V$ is indecomposable. 
In this special case we prove something slightly stronger: 
we show that $\mathcal{S}_g$ spans $V$ and, either 
$\mathcal{S}_g-\mathcal{S}_g:=\{v-w\mid v,w\in\mathcal{S}_g\}$ also spans $V$, 
or $q=2$ and $g$ is unipotent.  
Let $p$ be the characteristic of the field $\mathbb{F}_q$. Write $g=xu=ux$ with $x$ semisimple and with $u$
unipotent. Since $V$ is an indecomposable module,  $g$ is a cyclic matrix 
by \cite[Theorem 11.7]{HH}, and $g$ is 
conjugate, see for example~\cite[Lemma~$4.2$]{AIPS}, to a matrix (its Jordan 
canonical form)
\begin{equation}\label{eqmat}
\left(
\begin{array}{cccccc}
A&I&&      &\\
 &A&&     &\\
 & &\ddots&\ddots&\\
 & &&     A&I\\
 & &&      &A\\
\end{array}
\right)
\end{equation}
where $A$ and $I$ are $m\times m$-matrices over $\mathbb{F}_q$ (for
some divisor $m$ of
$d$), $I$ is the identity matrix and $A$ is the matrix induced by the
semisimple part $x$. Since $V$ is indecomposable, the
matrix $A$ acts irreducibly on $\mathbb{F}_q^m$. Furthermore, 
$m=d$ if and only if $g=x$ is semisimple. Replacing
$g$ by a conjugate, if necessary, we may assume that $g$ is as in~\eqref{eqmat}.
Since $A$ acts
irreducibly on $\mathbb{F}_q^m$, using Schur's lemma, we see that the
action of $A$ on $\mathbb{F}_q^m$ is equivalent to the action by
multiplication of a non-zero scalar of $\mathbb{F}_{q^m}$ on
the extension field $\mathbb{F}_{q^m}$.  Therefore every orbit of $A$ on
$\mathbb{F}_q^m\setminus\{0\}$ has length $|A|$.
Thus, if $m=d$, then $\mathcal{S}_g=V\setminus\{0\}$ and in this case both
$\mathcal{S}_g$ and $\mathcal{S}_g - \mathcal{S}_g$ span $V$. 
Hence we may assume that $m<d$, and so 
$|g|=p|A|$ since $|g|$ is square-free. 

Write $c=d/m>1$. In view of the shape of the matrix $g$, we write the
vectors of $V$ with only $c$ coordinates, where every coordinate is an 
element of $\mathbb{F}_{q^m}$.
A direct computation shows that
 $\mathcal{S}_g$ consists of all the
vectors of the form $v=(v_1,\ldots,v_c)$ with $(v_1,\ldots,v_{c-1})\neq
(0,\ldots,0)$. Now if $v\in\mathcal{S}_g$ then $\lambda v\in\mathcal{S}_g$ 
for all $\lambda\in \mathbb{F}_{q^m}\setminus\{0\}$. Note that if $q^m\neq 2$, 
then there exists $\alpha\in\mathbb{F}_{q^m}\backslash\{0,1\}$ and so for 
each $v\in \mathcal{S}_g$ we have $\alpha v$ and $(\alpha-1)v$ in $\mathcal{S}_g$. 
Thus $v=\alpha v-(\alpha-1)v\in \mathcal{S}_g-\mathcal{S}_g$. 
Hence $\mathcal{S}_g\subseteq \mathcal{S}_g-\mathcal{S}_g$, except possibly 
when $q^m=2$, that is, $p=2$, and $m= 1$. In particular, 
for $p\neq 2$, and for $m\ne 1$, our claim follows because clearly 
$\mathcal{S}_g$ spans $V$. For $p=2$ and $m=1$, the set $\mathcal{S}_g$ 
still spans $V$ and the element $g=u$ is cyclic unipotent of order $p$ 
and the proof of the claim is complete. In particular the inductive step is proved when $s=1$.

Assume now that $s>1$. Then by the claim above, for each $i$, 
$W_i=\langle\mathcal{S}_i\rangle$ and also, either 
$W_i=\langle\mathcal{S}_i-\mathcal{S}_i\rangle$, or $q=2$ 
and $g_i$ is unipotent. Suppose that $t$ of the elements $g_1,\ldots,g_s$ are unipotent, 
where $0\leq t\leq s$. If $t>0$ then, relabelling the index set, we may assume that 
$g_1,\ldots,g_t$ are unipotent, that is, the action of $g$ on each 
of $W_1,\ldots,W_t$ is unipotent. Since $g$ has square-free order, 
we have $|g_1|=\cdots=|g_t|=p$.
If $t=0$ set  $i_0=0$, $w_0=0\in V$ and $\mathcal{S}_{i_0}=\{0\}$, 
and if $t>0$ choose $i_0\in \{1,\ldots,t\}$ and $w_{i_0}\in 
\mathcal{S}_{i_0}$. Moreover, for each  $i\in \{t+1,\ldots,s\}$, choose $w_i\in
\mathcal{S}_i$.  Set $v=w_{i_0}+w_{t+1}+w_{t+2}+\cdots +w_s$, and write
$\ell=|v^{\langle g\rangle}|$. Since $v^{g^\ell}=v$ and since
$V=W_1\oplus \cdots\oplus W_s$ is a
$g$-invariant decomposition, we must have $w_i^{g^\ell}=w_i$, for each
$i\in \{t+1,\ldots,s\}$ and also for $i=i_0$. 
However, for each $i>t$, $w_i^{g^\ell}=w_i^{g_i^\ell}$ and hence
$\ell$ is divisible by $|g_i|$ by our choice of $w_i$. 
Also, if $t>0$ then $w_{i_0}^{g^\ell}=w_{i_0}^{g_{i_0}^\ell}$,
and hence, in this case, $\ell$ is divisible also by $|g_{i_0}|=p$. It
follows from~\eqref{aff:1}, for any $t$, that $|g|$ divides $\ell$. 
This shows that $|g|=|v^{\langle
  g\rangle}|$, that is, $v\in \mathcal{S}_g$.  In particular, we have
proved that
$$
\Delta:=\bigcup_{i_0\leq t}\left\{w_{i_0}+\sum_{i=t+1}^s w_i  
\mid w_{i_0}\in \mathcal{S}_{i_0}, w_i\in \mathcal{S}_i, \textrm{ for each
}i\in \{t+1,\ldots,s\}\right\}\subseteq
\mathcal{S}_g.
$$
Fix $i>t$ and choose distinct elements $v, v'\in\Delta$ with all coordinates 
equal except in position $i$. Note that this is possible since $g_i$ is 
not unipotent and hence $W_i=\langle\mathcal{S}_i-\mathcal{S}_i\rangle$. 
Now $v-v'\in \mathcal{S}_i-\mathcal{S}_i$ 
and hence $\mathcal{S}_i-\mathcal{S}_i\subseteq \langle \Delta\rangle
\subseteq \langle \mathcal{S}_g\rangle$. Since $\mathcal{S}_i-\mathcal{S}_i$ 
spans $W_i$ we obtain that $W_i\subseteq \langle \mathcal{S}_g\rangle$. 
Thus $W_{t+1}\oplus\cdots \oplus W_s\subseteq \langle\mathcal{S}_g\rangle$. 
From this and from the description of the elements of $\Delta$, we 
also have $\mathcal{S}_{i_0}\subseteq \langle \mathcal{S}_g\rangle$, 
for every $i_0\leq t$, and the result follows by induction.
\end{proof}

With Lemma~\ref{GL}, the main result of this section follows easily.

\begin{theorem}\label{thrm:affine}
Each element of $\AGL_d(q)$ has a regular cycle in $V$.
\end{theorem}

\begin{proof}
The group $\AGL_d(q)$ is the semidirect product $V\rtimes \GL_d(q)$, where $V$ is a $d$-dimensional vector space of dimension $d$ over the field $\mathbb{F}_q$ with $q$ elements. In particular, every element of $\AGL_d(q)$ is a pair $(h,v)$ with $h\in \GL_d(q)$ and $v\in V$. We recall that the group element $(h,v)$ acts on $w\in V$ via $w^{(h,v)}=w^h+v$.

Now, $\AGL_d(q)$ can be identified with a subgroup of $\GL_{d+1}(q)$. Indeed, the mapping

\[
(h,v)\mapsto\left(
\begin{array}{ccc}
h&0\\
v&1\\
\end{array}
\right)
\]
defines an isomorphism of $\AGL_d(q)$ onto a subgroup $G$ of
$\GL_{d+1}(q)$. Observe that the action of $\AGL_d(q)$ on $V$ is
equivalent to the action of $G$ on the vectors
$(w,1)\in\mathbb{F}_q^{d+1}$ having the last coordinate equal to $1$.

Let $g'$ be the element of $G$ corresponding to an element $g\in \AGL_d(q)$. From
Lemma~\ref{GL}, we see that there exists a vector $(w,\lambda)\in
\mathbb{F}_q^{d+1}$ with $|(w,\lambda)^{\langle g'\rangle}|=|g'|$ and
with $\lambda\neq 0$.  Observing that, for $\mu\in
\mathbb{F}_q\setminus\{0\}$, we have
$(\mu(w,\lambda))^{g'}=\mu(w,\lambda)^{g'}$,  we see that by replacing
the element $(w,\lambda)$ by
$(\lambda^{-1}w,1)$ we may assume that $\lambda=1$. Now
the action of $g'$ on $(w,1)$ is
equivalent to the affine action of $g$ on $w$ and this shows that $w$ lies in  a
$g$-cycle of length $|g'|=|g|$.
\end{proof}

\section{Primitive groups of Diagonal type}\label{sec:diagonal}

We start by recalling the structure of the finite primitive groups 
of Diagonal type. This will also allow us to set up the notation for this
section.

Let $\ell\geq 1$ and let $T$ be a non-abelian simple group. Consider
the group $N=T^{\ell+1}$ and $D=\{(t,\ldots,t)\in N\mid t\in T\}$, a
diagonal subgroup of $N$. Set $\Omega:=N/D$, the set of 
right cosets of $D$ in $N$. Then $|\Omega|=|T|^\ell$. Moreover we
may identify each element $\omega\in \Omega$ with an element
of $T^\ell$ as follows: the right coset
$\omega=D(\alpha_0,\alpha_1,\ldots,\alpha_\ell)$
contains a unique element whose first coordinate is $1$, namely, the
element
$(1,\alpha_0^{-1}\alpha_1,\ldots,\alpha_0^{-1}\alpha_\ell)$. We
choose this
distinguished coset representative and we denote the element 
$D(1,\alpha_1,\ldots,\alpha_\ell)$ of $\Omega$
simply by
$$
[1,\alpha_1,\ldots,\alpha_\ell].
$$
Now the element $\varphi$ of $\Aut(T)$ acts on $\Omega$ by
$$
[1,\alpha_1,\ldots,\alpha_\ell]^\varphi=[1,\alpha_1^\varphi,\ldots,\alpha_\ell^\varphi].
$$
Note that this action is well-defined because $D$ is
$\Aut(T)$-invariant.
Next, the element $(t_0,\ldots,t_\ell)$ of $N$ acts on $\Omega$ by
$$
[1,\alpha_1,\ldots,\alpha_\ell]^{(t_0,\ldots,t_\ell)}=
[t_0,\alpha_1t_1,\ldots,\alpha_\ell
  t_\ell]=[1,t_0^{-1}\alpha_1t_1,\ldots,t_0^{-1}\alpha_\ell t_\ell].
$$
Observe that the action induced by  $(t,\ldots,t)\in N$ on $\Omega$ is
the same as the action induced by the inner automorphism 
corresponding to conjugation by $t$.
Finally, the element $\sigma$ in $\Sym(\{0,\ldots,\ell\})$  acts
on $\Omega$ simply by permuting the coordinates. Note that this action
is well-defined because $D$ is $\Sym(\ell+1)$-invariant.

The set of all permutations we described generates a group $W$ isomorphic
to $T^{\ell+1}\cdot (
\Out(T)\times \Sym(\ell+1))$. A subgroup $G$ of $W$ containing the socle 
$N$ is primitive if either $\ell=2$ or $G$  acts primitively by 
conjugation on the $\ell+1$ simple direct factors of $N$~\cite[Theorem~4.5A]{DM}. Such 
primitive groups are the primitive groups of Diagonal type.
 Write 
$$
M=\{(t_0,t_1,\ldots,t_\ell)\in N\mid t_0=1\}.
$$ 
Clearly, $M$ is a normal subgroup of $N$ acting
regularly on $\Omega$. Since the stabilizer in $W$ of the point
$[1,\ldots,1]$ is $\Sym(\ell+1)\times \Aut(T)$, we obtain
that 
$$
W=(\Sym(\ell+1)\times \Aut(T))M.
$$ 
Moreover,  every element $x\in W$
can be written uniquely as $x=\sigma\varphi m$, with $\sigma\in \Sym(\ell+1)$,
$\varphi\in\Aut(T)$ and $m\in M$.

We first prove a sequence of lemmas about the fixed-point-ratio of
 elements of diagonal groups. We use the notation  above.

\begin{lemma}\label{SDlem1}
Let $x=\varphi m \in W$, with $x\neq 1$, $\varphi\in \Aut(T)$ and $m=(1,t_1,\ldots,t_\ell) \in M$.
Then $\mathrm{fpr}_\Omega(x)\leq \frac{1}{m(T)^\ell}$ where $m(T)$ is the minimal 
degree of a faithful permutation representation of $T$.
\end{lemma}

\begin{proof}
 Let $\omega=[1,\alpha_1,\ldots,\alpha_\ell]\in \Fix_\Omega(x)$.  Then
\begin{eqnarray*}
[1,\alpha_1,\ldots,\alpha_\ell]&=&
[1,\alpha_1,\ldots,\alpha_\ell]^x=
[1,\alpha_1^\varphi t_1,\ldots,\alpha_\ell^\varphi t_\ell ].
\end{eqnarray*}
This gives $\alpha_i^{\varphi}t_i=\alpha_i$, for each $i\in
\{1,\ldots,\ell\}$. In
particular, computing in $\Aut(T)$, we have, for each $i$, 
$$
\alpha_i^{-1}\varphi^{-1}\alpha_i=t_i^{-1}\varphi^{-1}.
$$
This implies that $\varphi^{-1}$ and  $t_i^{-1}\varphi$ are elements of
$\Aut(T)$ conjugate via the  element $\alpha_i$ of $T$. Now the number
of $\alpha_{i}\in T$ conjugating $\varphi^{-1}$ to
$t_{i}^{-1}\varphi^{-1}$ is either  
$|\cent T{\varphi^{-1}}|$ or $0$, according to whether $\varphi^{-1}$ and
$t_i^{-1}\varphi^{-1}$ are, or are not, in the same $T$-conjugacy class. Observe
that,  since we
are trying to obtain an upper bound on $\fpr_\Omega(x)$,
we may assume that $\varphi\neq 1$ because otherwise $x$ is the
translation by $m\in M$, which has no fixed points. This shows
that
\begin{eqnarray}\label{eq:SD1}
\mathrm{fpr}_\Omega(x)&=&
\frac{
|\Fix_\Omega(x)|
}{
|\Omega|
}\leq \frac{|\cent T
  {\varphi^{-1}}|^\ell}{|T|^{\ell}}=\frac{1}{(|T:
\cent T{\varphi^{-1}}|)^\ell}\leq \frac{1}{m(T)^\ell}.
\end{eqnarray}
\end{proof}

It is very important to observe (from the proof) that $\fpr_\Omega(x)>0$ 
in Lemma \ref{SDlem1} only if $|\varphi|$, and hence $|x|$, is divisible by a 
prime divisor of $|\Aut(T)|$.

\begin{lemma}\label{SDlem2}
Let $x=\sigma\varphi m\in W$, with $\sigma\in \Sym(\ell+1)\backslash\{1\}$ such that $0^\sigma=0$,
$\varphi\in \Aut(T)$ and $m=(1,t_1,\ldots,t_\ell) \in M$, and assume
that $p:=|x|$ is prime. Then $\mathrm{\fpr}_\Omega(x)\leq \frac{1}{|T|^{p-1}}$.
\end{lemma}

\begin{proof}
Let $\omega=[1,\alpha_1,\ldots,\alpha_\ell]\in \Fix_\Omega(x)$.  
By assumption, $\sigma\ne1$ and hence $|\sigma|=p$. 
Thus, relabelling the index set $\{1,\ldots,\ell\}$ if
necessary, we may write $\sigma=(1,\ldots,p)\cdots
((k-1)p+1,\ldots,kp)$, for some $k\geq 1$.
Then
\begin{eqnarray*}
[1,\alpha_1,\alpha_2,\ldots,\alpha_{p-1},\alpha_{p},\ldots]&=&
[1,\alpha_1,\alpha_2,\ldots,\alpha_{p-1},\alpha_p,\ldots]^x\\
&=&[1,\alpha_p^\varphi t_1,\alpha_1^\varphi t_2,\ldots,\alpha_{p-2}^\varphi
  t_{p-1},\alpha_{p-1}^\varphi t_p,\ldots].
\end{eqnarray*}
By considering the $2^{\textrm{nd}}$ coordinate, we get
$$\alpha_{1}=\alpha_p^\varphi t_{1}.$$ Now by considering the
$3^{\textrm{rd}}$ coordinate, we
obtain $$\alpha_{2}=\alpha_{1}^\varphi
t_{2}=\alpha_p^{\varphi^2}t_{1}^\varphi t_{2}.$$
Proceeding inductively we see that, for $i\in \{1,\ldots,p-1\}$, we have

$$\alpha_i=\alpha_p^{\varphi^{i}}t_1^{\varphi^{i-1}}\cdots
t_{i-1}^\varphi t_i.$$
This yields that the $p-1$ entries $\alpha_1,\ldots,\alpha_{p-1}$
of $\omega$ are uniquely determined by $x$ and by 
$\alpha_p$. Thus
\begin{eqnarray}\label{eq:SD2}
\mathrm{\fpr}_\Omega(x)&=&\frac{|\Fix_\Omega(x)|}{|\Omega|}\leq \frac{|T|^{\ell-p}|T|}{|T|^{\ell}}=\frac{1}{|T|^{p-1}}.
\end{eqnarray}
\end{proof}

\begin{lemma}\label{SDlem3}
Let $x=\sigma\varphi m\in W$, with $\sigma\in \Sym(\ell+1)$ such that $0^\sigma\neq 0$,
$\varphi\in \Aut(T)$ and $m=(1,t_1,\ldots,t_\ell) \in M$, and assume
that $p:=|x|$ is prime. 
\begin{enumerate}
 \item If $p>2$ then $\mathrm{\fpr}_\Omega(x)\leq \frac{1}{|T|^{p-2}}$.
 \item If $p=2$ then $\mathrm{\fpr}_\Omega(x)\leq \frac{4}{15}$.
\end{enumerate}
\end{lemma}

\begin{proof}
Let $\omega=[1,\alpha_1,\ldots,\alpha_\ell]\in\Fix_\Omega(x)$.  
By assumption, $\sigma\ne1$ and hence $|\sigma|=p$. 
Relabelling the index set $\{1,\ldots,\ell\}$ if
necessary, we may write $\sigma=(0,1,\ldots,p-1)\cdots
((k-1)p,\ldots,kp-1)$, for some $k\geq 1$. Here the argument is a
little more delicate and in fact the case $p=2$ requires extra
care. We have
\begin{eqnarray}\label{long}\nonumber
\omega&=&
[1,\alpha_1,\alpha_2,\ldots,\alpha_{p-2},\alpha_{p-1},\ldots]^x=[\alpha_{p-1},1,\alpha_1,\alpha_2,\ldots,\alpha_{p-2},\cdots]^{\varphi m}\\\nonumber
&=&[1,\alpha_{p-1}^{-1},\alpha_{p-1}^{-1}\alpha_1,\ldots,\alpha_{p-1}^{-1}\alpha_{p-3},\alpha_{p-1}^{-1}\alpha_{p-2},\ldots]^{\varphi 
m}\\
&=&
[1,(\alpha_{p-1}^{-1})^\varphi t_1,(\alpha_{p-1}^{-1}\alpha_1)^\varphi
  t_2,\ldots,(\alpha_{p-1}^{-1}\alpha_{p-3})^\varphi
  t_{p-2},(\alpha_{p-1}^{-1}\alpha_{p-2})^\varphi t_{p-1},\ldots].
\end{eqnarray}
Suppose that $p>2$. By considering the $2^{\textrm{nd}}$ coordinate, we get
$$
\alpha_{1}=(\alpha_{p-1}^{-1})^\varphi t_{1},
$$ 
and from the $3^{\textrm{rd}}$ coordinate, we
obtain 
$$
\alpha_{2}=(\alpha_{p-1}^{-1}\alpha_1)^\varphi
t_{2}=(\alpha_{p-1}^{-1})^{\varphi}(\alpha_{p-1}^{-1})^{\varphi^2}t_{1}^\varphi t_{2}.
$$
Proceeding inductively we see that, for $i\in \{1,\ldots,p-2\}$,
the $p-2$ entries $\alpha_1,\ldots,\alpha_{p-2}$ 
of $\omega$ are uniquely determined by $x$ and by 
$\alpha_{p-1}$. Thus 
\begin{eqnarray}\label{eq:SD3}
\mathrm{fpr}_\Omega(x)&=&\frac{|\Fix_\Omega(x)|}{|\Omega|}\leq \frac{|T|^{\ell-(p-1)}|T|}{|T|^{\ell}}=\frac{1}{|T|^{p-2}}.
\end{eqnarray}

We now consider the case $p=2$, where~\eqref{eq:SD3} (although still
correct) becomes meaningless. Observe that the action of $x$ on the $2^{\textrm{nd}}$
coordinate is given by $\alpha\mapsto (\alpha^{-1})^\varphi t_1$. 
This map is an involution only if
\begin{equation}\label{eq:SD4}
\alpha=(t_1^{-1}\alpha^\varphi)^\varphi t_1,
\end{equation}
for each $\alpha\in T$. By choosing $\alpha=1$, we get $t_1^\varphi=t_1$
and~\eqref{eq:SD4} becomes
$$
\alpha=t_1^{-1}\alpha^{\varphi^2}t_1=\alpha^{\varphi^2 t_1},
$$
for every $\alpha\in T$. Therefore $\varphi^2$ acts as conjugation by
$t_1^{-1}$.
 We are now ready to bound the number of fixed points of $x$. Consider
the $2^{\textrm{nd}}$ coordinate of $\omega$ in~\eqref{long}. For
$\omega=\omega^x$, we need to have
$\alpha_1=(\alpha_1^{-1})^{\varphi }t_1$. We now perform some computations
in $\Aut(T)$. Recalling that $\varphi^2$ squares to $t_1^{-1}$ in
$\Aut(T)$, we get
\begin{equation}\label{inv}
1=\alpha_1^{-1}\varphi^{-1}\alpha_1^{-1}\varphi
t_1=\alpha_1^{-1}\varphi^{-1}\alpha_1^{-1}\varphi^{-1}\varphi^2 t_1=(\alpha_1^{-1}\varphi^{-1})^2t_1^{-1}t_1=(\alpha_1^{-1}\varphi^{-1})^2.
\end{equation}
Observe, that if $\alpha_1$ and $\alpha_1'$ are two distinct solutions
of~\eqref{inv}, then
$(\alpha_1^{-1}\varphi^{-1})^2=(\alpha_1'^{-1}\varphi^{-1})^2=1$ and
so $\langle
\alpha_1^{-1}\varphi^{-1},\alpha_1'^{-1}\varphi^{-1}\rangle$
is a dihedral group. Therefore the involution
$\alpha_1^{-1}\varphi^{-1}$ must invert
$$
(\alpha_1^{-1}\varphi^{-1})(\alpha_1'^{-1}\varphi^{-1})^{-1}=\alpha_1^{-1}\alpha_1'.
$$ 
This shows that, for a given solution $\alpha$ of~\eqref{inv},  all 
other solutions are of the form $\alpha'=\alpha e$, where $e\in T$ is inverted by
$\alpha^{-1}\varphi^{-1}$.

 From~\cite{Potter}, we see that an automorphism of a non-abelian
simple group $T$ cannot invert more than $4/15$ of the
elements of $T$, (and equality holds only for $T=\Alt(5)$). 
Therefore, we have at most $4|T|/15$ solutions to~\eqref{inv}.
In summary, we have

\begin{eqnarray}\label{eq:SD5}
\mathrm{fpr}_\Omega(x)&=&\frac{|\Fix_\Omega(x)|}{|\Omega|}\leq
\frac{|T|^{\ell-1}(4|T|/15)}{|T|^{\ell}}=\frac{4}{15}.
\end{eqnarray}
\end{proof}

Before proving our main result we need a rather technical number theoretic lemma. Recall that
in elementary number theory $\omega(n)$ denotes the number of distinct
prime divisors of the positive integer $n$. The following lemmas may be found 
in~\cite[Theorem~$13$]{Robin} and~\cite{Mass} respectively.

\begin{lemma}\label{numbertheory}
For $n\geq 26$, we have $\omega(n)\leq \log(n)/(\log(\log (n))-1.1714)$.
\end{lemma}

\begin{lemma}\label{mass}
Let $m\geq 3$  and let $a$ be the maximum order of an element of
$\Sym(m)$.  Then $$\log(a)\leq \sqrt{m\log m}\left(1+\frac{\log(\log(m))-0.975}{2\log(m)}\right).$$
\end{lemma}

We can now prove the following theorem.
\begin{theorem}\label{thrm:SD}
Let $G$ be a primitive group of Diagonal type. Then each element of $G$
has a regular cycle.
\end{theorem}

\begin{proof}
We use the notation introduced above, and we
assume without loss of generality that $G=W$ with socle $T^{\ell+1}$. Let $g\in W$.
By Lemma~\ref{basic}, we may assume that $g$ has square-free 
order $p_1\cdots p_r$ where the $p_i$ are pairwise distinct primes and $r\geq1$. 
If $g\in M\rtimes \Aut(T)$, then we set $h=g, k=1,$ and $s=r$.
If this is not the case then we may write $g=hk=kh$ where 
$\langle h\rangle :=  \langle g\rangle \cap (M\rtimes \Aut(T))$
and $k\ne1$. Relabelling
the index set $\{1,\ldots,r\}$ if necessary, we may assume that
$|k|= p_{s+1}\cdots p_r$ (the order of $g$ modulo $M\rtimes \Aut(T)$),
for some $s$ in the interval $[0,r)$, and $|h| = p_1\cdots p_s$
(or $h=1$ if $s=0$). Write
$x_i=g^{|g|/{p_i}}$ for $i\in \{1,\ldots,r\}$.

Observe that  if $i>s$ then $x_i$  is of the type given in 
Lemma \ref{SDlem2} or Lemma \ref{SDlem3}  (and
so $\fpr_\Omega(x_i)\leq 4/15$ if $p_i=2$ and $\fpr_\Omega(x_i)\leq
1/|T|^{p_i-2}$ if $p_i>2$), while if $i\leq s$ then $x_i$ is of the 
type given in Lemma \ref{SDlem1}  (and so $\fpr_\Omega(x_i)\leq 1/m(T)^\ell$).
Thus
\begin{eqnarray}\label{keyremark}
\sum_{i=1}^r\mathrm{fpr}_\Omega(x_i)&\leq&
\frac{s}{m(T)^\ell} + \frac{4}{15} + \sum_{\substack{s+1\leq i\leq r\\ p_i>2}}\frac{1}{|T|^{p_i-2}}.  
\end{eqnarray}
Clearly
$$
\sum_{u\geq 1}\frac{1}{|T|^u}=\frac{1}{|T|-1}\leq \frac{1}{59}
$$
and hence, from~\eqref{keyremark}, 
\begin{eqnarray}\label{stillgood}
\sum_{i=1}^r\mathrm{fpr}_\Omega(x_i)&\leq&\frac{s}{m(T)^\ell}+\frac{4}{15}+\frac{1}{59}.
\end{eqnarray}

Recall that $s$ is at most the number of prime divisors of $|\Aut(T)|$
and hence $s\leq \omega(|\Aut(T)|)$. We then have the inequality

\begin{eqnarray}\label{crude}
\sum_{i=1}^r\mathrm{fpr}_\Omega(x_i)&\leq&\frac{\omega(|\Aut(T)|)}{m(T)^\ell}+\frac{4}{15}+\frac{1}{59}.
\end{eqnarray}
Observe that~\cite[Table~$5.3A$]{KL} contains an explicit value of
$m(T)$ for each non-abelian simple group of Lie type. (Unfortunately
there are some inaccuracies in~\cite[Table~$5.3A$]{KL}, an amended table can be
found in~\cite[Table~$4$]{GMPSorders}.) Moreover, $m(T)$
for the $26$ sporadic simple groups can be extracted from~\cite{ATLAS}.
It is then a tedious computation, going through the list of the
finite non-abelian
simple groups and using the upper bound for
$\omega(n)$ in Lemma~\ref{numbertheory},  to check that the
right hand side of~\eqref{crude} is strictly less than $1$  except for 
the cases where $\ell=1$ and either
$T\cong\Alt(m)$ or $T\cong\PSL_2(q)$ (for $q\leq 11$). In
particular, apart from these exceptions, the theorem follows from
Lemma~\ref{lemma:apeman}.

Suppose next that $T\cong \Alt(m)$ and $\ell=1$. 
If $m=6$ then the right hand side of \eqref{crude} is 
$\frac{1}{2}+\frac{4}{15}+\frac{1}{59}<1$, so we may assume that
$m\ne6$ and hence that $\Aut(T)=\Sym(m)$.
Clearly, $s\leq \log(|g|)$. Let
$a$ be the maximum order of an element of $\Aut(T)$. 
We may alternatively write $g$ in the form $g=x\varphi\sigma$, 
with $x=(x_1,x_2)\in T^2$, $\varphi\in\{1, (12)\}\subset \Aut(T)$,
and $\sigma\in\Sym(2)$. If $\sigma=1$ then $g\in T\rtimes\Aut(T)$ 
and so has order at most $a^2$, while if $\sigma=(12)$, then  
$g^2=(x_1 x_2^\varphi, x_2 x_1^\varphi)$, and since 
$x_1 x_2^\varphi =(x_2^{-1}(x_2x_1^\varphi)x_2)^\varphi$, it follows that
$|g^2|=|x_1x_2^\varphi|\leq a$, so $|g|\leq 2a$. Thus in both cases
$|g|\leq a^2$ and hence $s\leq\log(|g|)\leq 2\log(a)$. Also, from 
Lemma~\ref{mass} we see that
$$\log(a)\leq \sqrt{m\log m}\left(1+\frac{\log(\log(m))-0.975}{2\log(m)}\right)$$
for every $m\geq 3$. Now using this new upper bound for $s$, we see
with another tedious computation that the right hand side
of~\eqref{stillgood} is strictly less than $1$ for every $m\geq 27$.

The remaining cases (that is, $T\cong \Alt(m)$ for $m\leq 26$ and
$T\cong \PSL_2(q)$ for $q\leq 11$) can be easily checked by hand by
computing explicitly the number of prime divisors of $\Aut(T)$ and by
using this upper bound on $s$ in~\eqref{stillgood}. In all cases, the
right hand side of~\eqref{stillgood} is strictly less than $1$ and
hence the proof follows from Lemma~\ref{lemma:apeman}.
\end{proof}

\noindent\emph{Proof of Corollary~\ref{cor1.5}.}\quad 
Let $T$ be a non-abelian simple group and $\sigma\in\Aut(T)$. Then 
$G=(T\times T)\cdot\langle\sigma\rangle\leq T\rtimes \Aut(T)$
is a primitive group of Diagonal type, as above, with 
$\ell=1$, and the element $\sigma$ of $\Aut(T)$ induces its 
natural action on $T$. By Theorem \ref{thrm:SD}, 
$\sigma$ has a cycle of length $|\sigma|$ on $T$.

\section{Primitive almost simple groups}\label{sec:AStype}

In this section we consider the primitive almost simple groups $G\leq\Sym(\Omega)$, 
apart from the classical groups which are dealt with in \cite{GS}.
That is to say, we assume that $T\vartriangleleft G\leq\Aut(T)$ for a non-abelian 
simple group $T$ which is not a classical group. We
subdivide the proof according to whether $T$ is a sporadic
group, an exceptional group of Lie type or an alternating group.

We start by specifying our notation. We denote
by $\Irr(G)$ the set of complex
irreducible characters of $G$. We recall that the \emph{principal character}
$\chi_0$ of $G$ is defined by $\chi_0(g)=1$, for each $g\in G$. (It is sometimes also denoted $\chi_0=1_G$.)
Let $\eta$ be a complex character of $G$ and let $\chi\in\Irr(G)$. 
We say that $\chi$ is a
\emph{constituent} of $\eta$ if the inner product $\langle
\eta,\chi\rangle_G\neq 0$. 
%
%
The following lemma is well-known and we refer
to~\cite[Lemma~$2.7$]{LS} for the statement closest to our needs.

\begin{lemma}\label{obvious2}
Let $G\leq\Sym(\Omega)$ and let $g\in
G$. Assume that the derived subgroup $G'$ of $G$ is transitive.
Then, there exists a non-principal constituent $\chi$ of the
permutation character of $G$ such that 
$$
\mathrm{fpr}_\Omega(g)\leq\frac{1+|\chi(g)|}{1+\chi(1)}.
$$
\end{lemma}

\subsection{The sporadic simple groups}\label{sec:sporadic}
The main result of this section is the following. Our proof
makes use of information in~\cite{ATLAS}, and some
calculations with the computer algebra system
\textsc{Magma}~\cite{magma}.

\begin{theorem}\label{thrm:sporadic}Let $G$ be a
  primitive almost simple group on $\Omega$ with  socle a sporadic simple
  group. Then each element of $G$ has a regular cycle in $\Omega$.
\end{theorem}
\begin{proof}
Let $T$ be the socle of $G$.  Our proof is a case-by-case analysis for each
sporadic simple group $T$. Let $g\in G$. Write $n=|g|$ and recall that 
$\omega(n)$ is the number of distinct
prime divisors of $n$. By Lemma~\ref{basic}, we may assume that $n$
is square-free. Moreover, we may assume that $n$ is not a prime since every
element of prime order $n$ has a cycle of length $n$. Note that
the number $\omega(|g|)$ can be easily
obtained from~\cite{ATLAS}: and, in particular, in each case,
$\omega(|g|)\leq 3$.
Set 
$$
a_0:=\max\left\{\frac{1+|\chi(x)|}{1+\chi(1)}\mid x\in
G\setminus\{1\},\chi\in \Irr(G)\setminus\{\chi_0\}
\right\}.
$$ 
Suppose first that
\begin{equation}\label{eqthrmspor}
\omega(n)\,a_0<1.
\end{equation}
For each prime $p$ with $p\mid n$, let $\chi_{p}$ be a non-principal irreducible
constituent of the permutation character of $G$, as in Lemma~\ref{obvious2}, such that
$$
\mathrm{fpr}_\Omega(g^{n/p})\leq \frac{1+|\chi_p(g^{n/p})|}{1+\chi_p(1)}.
$$
Then
\begin{eqnarray*}
\sum_{\substack{p\mid n\\p\,\mathrm{prime}}}\mathrm{fpr}_\Omega(g^{n/p})\leq
\sum_{\substack{p\mid
    n\\p\,\mathrm{prime}}}\frac{1+|\chi_p(g^{n/p})|}{1+\chi_p(1)}\leq \omega(n)a_0<1
\end{eqnarray*}
and by Lemma~\ref{lemma:apeman}, the element $g$ has a cycle of length $n$.

A direct inspection in~\cite{ATLAS} shows that~\eqref{eqthrmspor}
always holds except when $\omega(n)=3$ and $G$ is one of the groups
$Co_2, Fi_{22}$ or $Fi_{22}:2$. Thus we may assume that $G$ is one of these three groups
and that $\omega(n)=3$. The elements of square-free
order $n$ with $\omega(n)=3$ in these groups either have order $30$, 
or in the case of $Fi_{22}:2$, they may have order $30$ or $42$. Moreover, for
each of these groups there exists a unique irreducible character (denoted
by $\chi_2$ in~\cite{ATLAS}) with
$$
3\frac{1+|\chi_2(x)|}{1+\chi_2(1)}\geq 1,
$$
for some $x\in G\setminus\{1\}$. In particular, if $\chi_2$ is not a
constituent of the permutation character of $G$, then we can use the previous argument (where
the maximum in $a_0$ runs through the irreducible characters different
from $\chi_2$) and we obtain that $g$ has a regular cycle. 
Therefore we may assume that $\chi_2$ is a constituent of the
permutation character $\pi$ of $G$.

Let $H$ be the stabilizer of a point of $\Omega$. From Frobenius reciprocity we
have
\begin{equation}\label{eq2}
0\neq \langle \pi,\chi_2\rangle_G=\langle
1_{H}^G,\chi_2\rangle_G=\langle 1_H,(\chi_2)_H\rangle_H.
\end{equation}
Now the \textsc{magma} libraries of complex characters have the good
taste to contain, for every maximal subgroup  $M$ of $G$, the
irreducible complex characters of $M$. In particular, for each $G\in
\{Co_2,Fi_{22},Fi_{22}:2\}$ and for each maximal subgroup $M$ of $G$,
we can compute $\langle 1,(\chi_2)_M\rangle$. This number is not $0$
only in the cases described in Table~\ref{table1}.

\begin{table}[!ht]
\begin{center}
\begin{tabular}{|l|l|}\hline
$G$&$M$\\\hline
$Co_2$&$M_{23}$ or $McL$\\
$Fi_{22}$&$2^{10}.M_{22}$ or $M_{12}$\\
$Fi_{22}:2$&$2^{10}.M_{22}.2$ or $M_{12}.2$\\\hline
\end{tabular}
\end{center}
\caption{Maximal subgroups $M$ with $\langle 1_M,(\chi_2)_M\rangle\neq 0$}\label{table1}
\end{table}

So in view of~\eqref{eq2}, we may assume that $H$ is one of the
maximal subgroups $M$ in Table~\ref{table1}. Using the information in~\cite{ATLAS}, 
in each of these six cases we can construct $G$ and $H$ in \textsc{Magma} and find 
representatives of each of the conjugacy classes of elements of order $30$ and $42$ 
in $G$. For each representative $g$ and for one of the generators $x$ of $G$ given 
by \cite{ATLAS} we see that $Hxg^i\neq Hxg^j$ for all distinct $i,j\in\{1,\ldots,|g|\}$. 
Thus $g$ has a regular cycle.


\end{proof}

\subsection{Exceptional groups of Lie type}\label{section}

For exceptional groups of Lie type Lawther, Liebeck and
Seitz~\cite[Theorem~$1$]{LLS} have obtained
useful  and explicit upper bounds on
$\fpr_\Omega(x)$. With their result we can prove the following theorem.

\begin{theorem}\label{thrm:excep}Let $G$ be a finite primitive group on $\Omega$ with
  socle an exceptional simple group of Lie type. Then each element of $G$ has a regular cycle on $\Omega$.
\end{theorem}

\begin{proof}
Let $T$ be the socle of $G$, let $g\in G$, and let $m(T)$ be the minimum 
degree of a faithful permutation representation of $T$.
By Lemma~\ref{basic}, we may assume that $g$ has square-free order
$p_1\cdots p_s$. Now, $s\leq \log_2(|g|)\leq \log_2(m(T))$ (where the
last inequality follows from the main result of~\cite{GMPSorders}).  A comprehensive table
containing $m(T)$ for each exceptional simple group $T$ of Lie type 
can be found in~\cite[Table~$4$]{GMPSorders}.

Now the proof follows easily from
Lemma~\ref{lemma:apeman} by a case-by-case analysis. We discuss
here with full details the case  $T=E_8(q)$. We see that
$m(T)=(q^{30}-1)(q^{12}+1)(q^{10}+1)(q^6+1)/(q-1)$ and so $m(T)\leq
q^{58}$. Also from~\cite{LLS}, we see that $\fpr_\Omega(x)\leq
1/(q^8(q^4-1))$, for every $x\neq 1$. Now the inequality
$$\frac{\log_2(q^{58})}{q^8(q^4-1)}<1$$
is satisfied for all $q$. All the other cases are similar.

The only simple groups $T$  where this approach does not work are
$E_6(2)$, $F_4(2)$, ${}^3D_4(2)$, $G_2(3)$ and
$G_2(4)$. However, we can see from~\cite{ATLAS} that an element of
$\Aut(T)$ is at most the product of $8$, $3$, $2$,
$2$, $2$, distinct primes, respectively. With this new upper bound on $s$
and using~\cite{LLS}
it is straighforward to see that $s\fpr_\Omega(x)<1$, for every $x\neq
1$. Now the proof follows as usual from Lemma~\ref{lemma:apeman}.
\end{proof}

\subsection{The Alternating Groups}\label{alt}
Suppose that $G$ is almost
simple with socle $\Alt(m)$, for some $m\geq
5$. In most cases $G\leq \Sym(m)$, and we make a brief 
formal comment about the situation where this does not hold. 

\begin{remark}\label{rem}
{\rm There are three primitive actions of groups $G$ satisfying
$\Alt(6)\leq G\leq\Aut(\Alt(6))$ but $G\not\leq\Sym(6)$. The groups are  
$M_{10}, \PGL_2(9)$, and $\Aut(A_6)$ and they all have primitive actions of degrees 
$10, 36, 45$. A simple \textsc{Magma} computation shows that Conjecture~\ref{conj} 
is true for all these primitive actions.
}
\end{remark}

From now on we assume that $G=\Alt(m)$ or $\Sym(m)$.
First we prove Theorem~\ref{actiononsets} which shows that the action of $\Sym(m)$ on
$k$-sets is a genuine exception in Theorem~\ref{thrm:main}.

\begin{proof}[Proof of Theorem \ref{actiononsets}]
Let $k\leq m/2$.
Denote by $p_1,\ldots,p_{k+1}$ the first $k+1$ primes,
and let $n:=n_k=\sum_{i\leq k+1}p_i$.
Suppose first that $m\geq n$. Let $g=\sigma_1\cdots \sigma_{k+1}$ be a permutation having $k+1$
disjoint cycles $\sigma_1,\ldots,\sigma_{k+1}$ of length 
$p_1,\ldots,p_{k+1}$, respectively (such an element exists since $m\geq n$). Now, let $X$ be a
$k$-subset of $\{1,\ldots,m\}$. Since $|X|=k$, we see that $X$
intersects at most $k$ of the supports of the cycles of $g$. By the
pigeon-hole principle, there exists a cycle $\sigma_i$ not
intersecting $X$. In particular $\sigma_i$ fixes $X$ point-wise and
hence $X^{\langle g\rangle}$ has size at most $|g|/p_i$. 
Since this arguments holds for all $X$, it follows that $g$ has no regular
cycle on $k$-sets. 

On the other hand, suppose that $m<n$ and let $g\in \Sym(m)$. Write
$g=\sigma_1\cdots\sigma_t$ with $\sigma_1,\ldots,\sigma_t$ the
disjoint cycles of $g$, so $|g|=\lcm\{|\sigma_i|\mid i\in
\{1,\ldots,t\}\}$. Since $m<n$,  $g$ can have at most $k$ cycles
having pairwise coprime lengths, and hence, relabelling the index set
$\{1,\ldots,t\}$ if necessary, we may assume that
$$
|g|=\lcm\{|\sigma_i|\mid i\in \{1,\ldots,s\}\},
$$
for some $s\leq k$.  Write $\ell_i=|\sigma_i|$, for 
$i\in \{1,\ldots,s\}$, and $\ell=\sum_{i=1}^s\ell_i$.
In particular,
as $s\leq k$, there exists a $k$-subset $X$ of $\{1,\ldots,m\}$ intersecting the
support of $\sigma_i$ in at least one point, for each $i\in \{1,\ldots,s\}$. Write $x_i$ for the size
of the intersection of $X$ with the support of $\sigma_i$; so $x_i>0$ for each $i\leq s$.

Suppose that $k\leq \ell-s$. Then for each $i\in \{1,\ldots,s\}$, we can choose $X$ 
so that $X$ intersects the support of
$\sigma_i$ in $x_i$ consecutive points with $x_i<\ell_i$ and with
$\sum_{i=1}^sx_i=k$. Now it is clear that $X^{\langle g\rangle}$ has
size $|g|$.

Next suppose that $k>\ell-s$ and that $m\geq k+s$. Choose a 
subset $X_i$ of size $\ell_i-1$
from the support of $\sigma_i$, for each $i\in
\{1,\ldots,s\}$. Then $\cup_i X_i$ has size $\sum_i(\ell_i-1)=\ell-s$.
Since $m\geq k+s$, we have $m-\ell\geq k-(\ell-s)$ and so there
exists a subset $Y$ of $\{1,\ldots,m\}$ of size $k-(\ell-s)$ disjoint
from the support of $\sigma_i$, for each $i\in
\{1,\ldots,s\}$. Then $X:=(\cup_iX_i)\cup Y$ is a $k$-set 
and $X^{\langle g\rangle}$ has
size $|g|$.

Finally suppose that $k>\ell-s$ and $m<k+s$. Recalling that $k\leq m/2$, 
by adding the inequalities $k>\ell-s$ and $k>m-s$, we obtain
$$
m\geq 2k> (\ell-s)+(m-s).
$$
Hence $2s>\ell$. However this is a contradiction because
$\ell=\sum_{i=1}^s\ell_i\geq \sum_{i=1}^s2=2s$.
\end{proof}

It is interesting to observe that $n_k$ is asymptotic to
$k^2\log(k)/2$, see~\cite{BS} (the rate of convergence is actually
rather slow).

\subsubsection{Partition actions}

There is another important action of the symmetric group 
that we need to study before moving to a general action. 
For $a,b\geq 2$, we say that a partition of the set $\{1,\ldots,ab\}$ into 
$b$ parts each of size $a$ is an \emph{$(a,b)$-uniform partition}.
The symmetric group $\Sym(ab)$ acts 
primitively on the set of $(a,b)$-uniform partitions, for every value of 
$a$ and $b$, and this action is faithful whenever $ab\neq 4$. (We note also that
the element $g=(1234)$ has cycles of lengths 1, 2 on the set of 
three $(2,2)$-uniform partitions, and hence has no regular cycle in this action.)

\begin{proposition}\label{uniformpartitions}
Let $a,b\geq 2$ such that $(a,b)\ne (2,2)$. Then every element of
$\Sym(ab)$ has a regular cycle on $(a,b)$-uniform partitions. 
\end{proposition}

We prove Proposition \ref{uniformpartitions} via a sequence of lemmas. 
The proof is similar to the proof of Theorem~\ref{actiononsets}, but 
unfortunately slightly more technical.  First we set up some notation.
Assume that $ab\ne 4$ and let $g\in\Sym(ab)$. Write 
$g=\sigma_1\cdots\sigma_t$ with $\sigma_1,\ldots,\sigma_t$ the disjoint 
cycles of $g$ on $\{1,\dots,ab\}$. Then $|g|=\lcm\{|\sigma_i|\mid i \in 
\{1,\ldots,t\}\}$. 
Relabelling the index set $\{1,\ldots,t\}$ if necessary, we may 
assume that $|g|=\lcm\{|\sigma_{1}|,\ldots,|\sigma_{s}|\}$, and that
$|g|$ is not the least common multiple of fewer than $s$ of the
$|\sigma_i|$. Denote by $L_i$ the support of $\sigma_i$, 
and let  $\ell_i=|L_i|$, for each $i\in \{1,\ldots,t\}$. 
Then $\ell=\sum_{i=1}^s\ell_i$. We may assume that $\sigma_1=
(1,\ldots,\ell_1)$, $\sigma_2=(\ell_1+1,\ldots,\ell_1+\ell_2)$, etc. 
Moreover we may reorder the cycles so that $\ell_1\leq \ell_2\leq \cdots \leq \ell_s$.

\begin{lemma}\label{lem1}
 If $ab\ne4$ and $g$ has a regular cycle on $\{1,\dots,ab\}$ (that is, $s=1$), 
then $g$ has a regular cycle on $(a,b)$-uniform partitions.
\end{lemma}
\begin{proof}
Note that $s=1$ implies that $\ell_i$ divides $\ell_1$, for every $i$,
and that $|g|=\ell_1$. If $\ell_1$ is a prime then $g$ has a regular cycle on
 $(a,b)$-uniform partitions since this action of $\Sym(ab)$ is faithful. 
Assume now that $\ell_1$ is not prime, so in particular $\ell_1\geq4$. 
We consider various cases, and for each we construct
a $g$-cycle of length $\ell_1$ on $(a,b)$-uniform partitions. 

Suppose first that $a\geq \ell_1$. Let $X_1$ be the first $\ell_1-1$ 
points of $L_1$ and $Y_1$ be the first $(a-\ell_1+1)$ points from 
$\{1,\ldots,ab\}\setminus L_1$. Set $A_1=X_1\cup Y_1$. Then $|A_1|=a$
and we extend $A_1$ to an $(a,b)$-uniform partition $\wp=\{A_1,\ldots,A_b\}$ 
of $\{1,\ldots,ab\}$ (in any way). Relabelling the index set $\{1,\ldots,b\}$ 
if necessary, we may assume that the last point $\ell_1$ of the support of 
$\sigma_1$ lies in $A_2$. Suppose that $\wp^{g^n}=\wp$, for some integer $n$.  
Then $A_1^{g^n}\in \wp$. Since the only elements of $\wp$ containing points 
from $L_1$ are $A_1$ and $A_2$, we obtain that either $A_1^{g^n}=A_1$ or 
$A_1^{g^n}=A_2$. In the first case, we must have $X_1^{g^n}=X_1^{\sigma_1^{n}}=X_1$. 
However, this happens only when $\ell_1$ divides $n$. In particular, $\wp$ is on a 
cycle of length $|g|$ of $g$. In the second case, since $A_2$ contains only one point 
from $L_2$, we have $\ell_1=2$, contradicting the fact that $\ell_1\geq4$.
 
We now deal with the case $a<\ell_1$. Write $\ell_1=aq+r$, with $q\geq 1$ 
and $0\leq r\leq a-1$. We consider the cases $r\geq 1$ and $r=0$ separately.  
Suppose first that $r\geq 1$, so that $|\wp|=b\geq q+1$. 
Consider the sets 
$$
A_i=\{(i-1)a+1,(i-1)a+2,\ldots,ia\}
$$ 
for $i\in \{1,\ldots,q+1\}$. Each of these sets has size $a$, and 
each of $A_1,\ldots,A_q$ is contained in $L_1$, while $A_{q+1}$ 
contains only $r<a$ points from $L_1$. 
Let $\wp$ be an $(a,b)$-uniform partition containing $A_1,\ldots,A_{q+1}$. 
Suppose that $\wp^{g^n}=\wp$, for some integer $n$. In particular, 
$A_{q+1}^{g^n}\in \wp$. Since $A_{q+1}$ is the only element of $\wp$ 
that has exactly $r$ points from $L_1$, we have $A_{q+1}^{g^n}=A_{q+1}$. 
Thus, as before, $\ell_1$ divides $n$ and the $g$-cycle on $(a,b)$-uniform 
partitions containing $\wp$ has length $|g|$.

  It remains to consider the case $r=0$, that is, $a\mid \ell_1$. 
We split this case according to $q<b$ or $q=b$. Suppose first that 
$q<b$, that is, $\sigma_1$ is not a cycle of length $ab$. Consider 
the sets $A_i=\{(i-1)a+1,(i-1)a+2,\ldots,ia\}$, for each $i\in \{1,
\ldots,q-1\}$, and define 
$$
\hspace{1.5cm} A_q=\{(q-1)a+1,(q-1)a+2,\ldots,qa-1,qa+1\}
$$ 
and 
$$
A_{q+1}=\{qa,qa+2,qa+3,\ldots,(q+1)a\}.
$$ 
Each of these sets has size $a$, the sets $A_1,\ldots,A_{q-1}$ are 
contained in $L_1$, and  $A_q$, $A_{q+1}$ contain $a-1$, $1$ points of $L_1$, 
respectively. Let $\wp$ be an $(a,b)$-uniform partition containing 
$A_1,\ldots,A_{q+1}$. Suppose that $\wp^{g^n}=\wp$, for some integer 
$n$. Then $A_{q+1}^{g^n}\in \wp$ and so either $A_{q+1}^{g^{n}}=A_{q+1}$ 
or $A_{q+1}^{g^n}=A_{q}$. In the first case, $(qa)^{g^n}=qa$ and $\ell_1$ 
divides $n$, so $\wp$ lies in a regular cycle for $g$. In the 
second case, we must have $a-1=1$ because $A_q$ contains $a-1$ 
points from $L_1$. Moreover, $A_q^{g^{n}}=A_{q+1}$ and hence
$g^n$ (and also $\sigma_1^n$) must interchange the points $qa-1=2q-1$ and $qa=2q$
(of $A_q\cap L_1$ and $A_{q+1}\cap L_1$). 
Now $\sigma_1=(1,2,\dots,2q)$, and the only way $\sigma_1^n$ can interchange 
the consecutive points $2q-1$ and $2q$ is if $q=1$ and $n$ is odd, but this implies that $\ell_1=aq=2$,
which is a contradiction.

It remains to consider the case $r=0$ and $q=b$, that is, $g=\sigma_1$ 
is the cycle $(1,2,\ldots,ab)$ of length $ab$. Suppose that $a>2$. 
Consider the $(a,b)$-uniform partition $\wp$ consisting of 
$A_1=\{2,3,\ldots,a-1,2a-1,2a\}$, $A_2=\{1,a,a+1,\ldots,2a-2\}$,  
and $A_i=\{(i-1)a+1,(i-1)a+2,\ldots,ia\}$, for $i\in \{3,\ldots,b\}$. 
Suppose that $\wp^{g^n}=\wp$. Then $A_1^{g^n}\in \wp$. Since each $A_i$ for 
$i\geq 3$ contains $a$ consecutive integers while $A_1$ does not, either 
$A_1^{g^n}=A_1$ or $A_1^{g^n}=A_2$. Suppose that the latter holds. Again 
looking at the size of the largest subsets 
of consecutive integers we see that this is only possible when $a-1=2$.  
We also require $A_2^{g^{n}}=A_1$, and so $g^n$ interchanges $A_1=\{2,5,6\}$ and $A_2=\{1,3,4\}$,
which is not possible for any $n, b$.  Thus $A_1^{g^n}=A_1$ and so $g^n$ fixes 
pointwise every element of $A_1$. This implies that $ab$ divides $n$, and $\wp$ lies in a regular $g$-cycle.
Finally suppose that $a=2$. Recall that $ab\neq 4$ and hence $b\geq 3$. 
We consider the  partition 
$$
\wp=\{ A_1=\{1,3\}, A_2=\{2,4\},\{5,6\},\ldots,\{2b-1,2b\}\}.
$$
Suppose that $\wp^{g^n}=\wp$, for some integer $n$. Then $A_1^{g^n}\in \wp$. 
A direct computation with $g$  gives that either $A_1^{g^n}=A_1$ or $A_1^{g^n}
=A_2$. In the first case, we have $1^{g^{n}}=1$ because $ab>4$, and so $ab$ 
divides $n$. In the second case, $A_2^{g^n}=A_1$, contradicting $ab>4$.
Thus $ab$ divides $n$ and $\wp$ lies in a regular $g$-cycle.
\end{proof}

\begin{lemma}\label{lem2}
If $ab\ne4$, $s\geq 2$ and $s\leq a\leq \ell-s$, 
then $g$ has a regular cycle on $(a,b)$-uniform partitions.
\end{lemma}

\begin{proof}
Since $s\geq2$ and $\ell_1\leq\ell_2\leq\dots\leq \ell_s$, 
we must have $\ell_1<\ell_2$ by the minimality of $s$.
For each $i\in \{1,\ldots,s\}$, let $X_i$ consist of the first $x_i$ 
consecutive points from the support of $\sigma_i$, where $1\leq x_i
\leq \ell_i-1$ and $a=\sum_{i=1}^s x_i$ (note that this is possible 
because of the restrictions on $a$). Suppose that $(x_1,\ldots,x_s)
\neq (1,\ldots,1)$ and $\ell_1\neq 2$. If $\ell_1=2x_1$  then 
either $(x_1,\ldots,x_s)=(\ell_1/2,1,\ldots,1)$ or there exists 
$i\in \{2,\ldots,s\}$ with $x_i>1$. Since $\ell_1\neq 2$ and $\ell_i>
\ell_1$ for all $i\geq 2$, in the first case we may remove a point from 
$X_1$ and add another point of $L_2$ to $X_2$, while in the second case 
we may add another point of $L_1$ to $X_1$ and remove a point from $X_i$. 
Proceeding in this way we may obtain $X_1,\ldots,X_s$ such that either 
$(x_1,\ldots,x_s)=(1,\ldots,1)$, or $\ell_1=2$, or $\ell_1\neq 2x_1$.

Write $A_1=\cup_{i=1}^sX_i$.
Now we describe how to complete $A_1$ to an $(a,b)$-uniform partition 
$\wp=\{A_1,A_2,\ldots,A_{b}\}$. We construct  $A_2,\ldots,A_{b}$ 
iteratively by induction: let $A_{b}$ consist of the $a$ largest 
points in $\{1,\ldots,ab\}\setminus A_1$, and if $A_j,\ldots,A_b$ have been constructed,
then for $A_{j-1}$ take the $a$ largest points in the set 
$$
\{1,\ldots,ab\}\setminus(A_1\cup A_j\cup A_{j+1}\cup\cdots \cup A_b).
$$ 
Suppose that $\wp^{g^n}=\wp$, for some integer $n$. Then $A_1^{g^n}\in \wp$. 
If $A_1^{g^n}=A_1$, then by our choice of $A_1$ we have $X_i^{g^n}=X_i$, 
and hence $\ell_i$ divides $n$, for every $i\in \{1,\ldots,s\}$. 
Thus $|g|$ divides $n$ and $\wp$ is in a regular $g$-cycle. Suppose then 
that $A_1^{g^n}=A_j$, for some $j\in \{2,\ldots,b\}$. We search for a 
contradiction.  This means that $A_j$ contains points from  $L_i$, for 
each $i\in \{1,\ldots,s\}$. However, by the way that we have constructed 
the partition $\wp$ we must have  $A_x\subseteq L_1$ for each $x\in 
\{2,\ldots,j-1\}$. We show that this forces $j=2$. Suppose that $j>2$. 
Then $A_2\subseteq L_1$ and hence $A_2^{g^n}$ is an element of $\wp$ 
contained in $L_1$. Thus $A_2^{g^n}=A_x$ for some $x\in \{2,\ldots,j-1\}$. 
The condition `$A_2,\dots,A_x$ contained in $L_1$' implies that
$A_x=\{x_1+(x-2)a+1,\ldots,x_1+(x-2)a+a\}$.
Since $g$ and $\sigma_1$ induce the same action on $L_1$, we see that
$A_2^{g^n}=A_2^{\sigma_1^n}=A_x$ if and only if $n\equiv (x-2)a
\pmod{\ell_1}$. This in turn imposes severe restrictions on $X_1$. 
In fact, $1^{g^n}=1^{\sigma_1^n}=1^{\sigma_1^{(x-2)a}}=(x-2)a+1\in 
A_{x-1}$ and as $1\in A_1$, we get $A_j=A_1^{g^n}=A_{x-1}$, a 
contradiction. Thus $j=2$. 
Moreover, from the way we have chosen $A_2$ we must have that 
$A_2$ contains $L_i\setminus X_i$, for each 
$i\in \{1,\ldots,s-1\}$, and at least $|X_s|$ points from $L_s\setminus 
X_s$. Furthermore, $A_1^{g^n}=A_2$ 
implies that $x_i=\ell_i-x_i$ for each $i\in \{1,\ldots,s-1\}$. 
Thus $\ell_i=2x_i$ for each $i\in\{1,\ldots,s-1\}$. Now our choice of $A_1$ 
comes into play, and shows that either $\ell_1=2$ or $(x_1,\ldots,x_s)=(1,\ldots,1)$. 
If $\ell_1=2$, the minimality of $s$ and the fact that $\ell_i=2x_i$ for all 
$i\leq s-1$ implies that $s=2$ (otherwise we can omit $\sigma_1$ in 
computing $\lcm\{|\sigma_i|\mid i\in \{1,\ldots,s\}\}$). Similarly, 
if $(x_1,\ldots,x_s)=(1,\ldots,1)$ then $\ell_i=2$ for all $i\leq s-1$; and
since $\ell_2>\ell_1$ we must again have $s=2$. Thus 
we have $\ell_1=s=2$.
In particular, $|g|=2\ell_2$, with $\ell_2$ odd. Now $A_1$ and $A_2$ 
have both one point in $L_1$ and $a-1$ points in $L_2$. Moreover, 
$\sigma_1=(1,2)$, $\sigma_2=(3,4,\ldots,\ell_2+2)$, $A_1=\{1,3,4,
\ldots,a+1\}$ and $A_2=\{2,a+2,a+3,\ldots,2a\}$. Now that we have 
the permutations $\sigma_1$ and $\sigma_2$ in our hands, with a direct 
computation we see that if $A_1^{g^n}=A_2$, then $n$ is odd and $n\equiv 
a-1\pmod {\ell_2}$. Since $n\leq |g|=2\ell_2$ this implies that either 
$n=a-1$ and $a-1$ is odd, or $n= a-1+\ell_2$ and $a-1$ is even. Observe, 
that since $\ell_2$ is odd, we cannot have $2a=\ell_2+2$. So there exist 
elements of $L_2$ not in $A_1$ or in $A_2$, and $b>2$. Now we look at $A_3$. 
First assume that $3a\leq \ell_2+2$, so that $A_3=\{2a+1,\ldots,3a\}$. We have 
$A_3^{g^n}\in \wp$. However, as $n\equiv a-1\pmod {\ell_2}$, we get 
that $(2a+1)^{g^n}=(2a+1)^{\sigma_2^{n}}=3a\in A_3$. Thus $A_3^{g^n}=A_3$, 
which is clearly a contradiction. Thus $3a>\ell_2+2$, so that $A_3\cap L_2=
\{2a+1,\ldots,\ell_2+2\}$ is properly contained in $A_3$. Since $A_3^{g^{n}}
\in \wp$, we must have $A_3^{g^n}=A_3$ (since the sets $A_1$ and $A_2$ contain 
only points from the first two cycles $\sigma_1$ and $\sigma_2$ of $g$). Since 
$A_3$ contains consecutive points of $L_2$, this implies that $\ell_2$ divides 
$n$. However, $n=a-1$ or $n=a-1+\ell_2$, neither of which is divisible by 
$\ell_2$ as $a-1<\ell_2$. Thus we have obtained a contradiction to the fact that $A_1^{g^n}\neq A_1$.
\end{proof}

\begin{lemma}\label{lem3}
If $ab\ne4$, $s\geq 2$ and $a> \ell-s$, 
then $g$ has a regular cycle on $(a,b)$-uniform partitions.
\end{lemma}
\begin{proof}
If also $a>ab-s$, then 
$$
ab\geq 2a>(\ell-s)+(ab-s)=ab+\ell-2s\geq ab,
$$
since $\ell=\sum_{i=1}^s\ell_i\geq 2s$, which is a contradiction.
Thus  $a\leq ab-s$. For each $i\in \{1,\ldots,s\}$, let $X_i$ consist of the 
first $\ell_i-1$ consecutive points from the support of $\sigma_i$
(so $\cup_{i=1}^sX_i$ has size $\ell-s<a$). Then take for $X$ the first 
$a-(\ell-s)$ points from 
$$
\{1,\ldots,ab\}\setminus(L_1\cup\cdots \cup L_s),
$$
observing that this is possible because $ab-\ell\geq a+s-\ell$. 
Set $A_1=(\cup_{i=1}^sX_i)\cup X$.
Now we complete $A_1$ to an $(a,b)$-uniform partition $\wp=\{A_1,A_2,
\ldots,A_{b}\}$. We construct  $A_2,\ldots,A_{b}$ iteratively by induction. 
Let $A_2$ consist of the $a$ smallest points of $\{1,\ldots,ab\}\setminus A_1$, 
and if $A_2,\ldots,A_j$ have been constructed, for $A_{j+1}$ take 
the $a$ smallest points in the set 
$$
\{1,\ldots,ab\}\setminus(A_1\cup A_2\cup\cdots \cup A_j).
$$ 
Suppose that $\wp^{g^n}=\wp$, for some $n$. Then $A_1^{g^n}\in \wp$. 
If $A_1^{g^n}=A_1$, then by our choice of $A_1$ we have that $X_i^{g^n}=X_i$ 
and hence $\ell_i$ divides $n$, for every $i\in \{1,\ldots,s\}$. So $\wp$ is in 
a regular $g$-cycle. If $A_1^{g^n}\neq A_1$, then our choice of $\wp$ gives 
$A_1^{g^n}=A_2$, since $A_2$ is the only part other than $A_1$ containing a 
point of $\sigma_1$. Now for every $i\in \{1,\ldots,s\}$, we have $(A_1\cap 
L_i)^{g^n}=A_2\cap L_i$, and hence $\ell_i-1=1$. This gives $\ell_i=2$, for every 
$i$, which contradicts the minimality of $s$ (recall that we are dealing with the case $s\geq 2$).
\end{proof}

\begin{lemma}\label{lem4}
If $ab\ne4$, $s\geq 2$ and $a< s$, 
then $g$ has a regular cycle on $(a,b)$-uniform partitions.
\end{lemma}
\begin{proof}
As above we define a suitable $(a,b)$-uniform partition 
$\wp=\{A_1,\ldots,A_b\}$ and show that $\wp$  is in a regular $g$-cycle. 
Write $s=aq+r$ with $q\geq1$ and $0\leq r<a$. For $i\in \{1,\ldots,q\}$, we let 
$A_i$ consist of the smallest element of $L_{(i-1)a+j}$, for each $j\in \{1,
\ldots,a\}$. If $r>0$, then we choose $A_{q+1}$ to consist of the smallest 
element of $L_{qa+j}$, for $j\in \{1,\ldots,r\}$, together with the smallest 
element of $L_{j}\setminus A_1$, for $j\in \{1,\ldots,a-r\}$. We define the  
remaining elements of $\wp$ by induction. Let $A_b$ consist of the $a$ largest 
elements which have not been assigned to any $A_i$ thus far. Then 
if $A_j,\ldots,A_b$ have been defined, let $A_{j-1}$ consist of the $a$ 
largest elements in 
$$
\{1,\ldots,ab\}\setminus(A_1\cup\cdots \cup A_q\cup A_j\cup\cdots \cup A_b)
$$
if $r=0$, or in
$$
\{1,\ldots,ab\}\setminus(A_1\cup\cdots \cup A_q\cup A_{q+1}\cup A_j\cup\cdots \cup A_b)
$$
if $r>0$. Suppose that $\wp^{g^n}=\wp$, for some integer $n$. Fix 
$i\in \{1,\ldots,q\}$. Then $A_i^{g^n}\in \wp$. Now $A_i$ contains 
at most one point from each cycle of $g$ and in fact, $A_i$ contains 
a point from  each of $\sigma_{(i-1)a+1},\sigma_{(i-1)a+2}\ldots,
\sigma_{ia}$.  By construction, $A_i$ is the only element of $\wp$ 
with this property (this can be easily seen by distinguishing the case 
$r=0$ and $r>0$, and by noticing that $\ell_i$ can be equal to $2$ only 
for $i=1$ and in this latter case all the other $\ell_i$ are odd). 
Thus $A_i^{g^n}=A_i$ and hence $\ell_{(i-1)a+j}\mid n$, for each $j\in 
\{1,\ldots,a\}$. If $r=0$, then this argument shows that  $\ell_x\mid n$ 
for each $x\in \{1,\ldots,s\}$ and hence $|g|\divides n$. Assume then that 
$r>0$ and consider $A_{q+1}^{g^{n}}$. This set contains a point from the first 
$a-r$ cycles of $g$ and a point from the cycles $\sigma_{qa+1},\sigma_{qa+2},
\ldots,\sigma_{qa+r}$. Again, by the way that $A_{q+2},\ldots,A_{b}$ were defined 
we must have $A_{q+1}^{g^n}=A_{q+1}$. Thus $\ell_{qa+j}\mid n$, for every $j\in \{1,\ldots,r\}$, and again $|g|\mid n$.
\end{proof}

Proposition \ref{uniformpartitions} now follows from Lemmas 
\ref{lem1}, \ref{lem2}, \ref{lem3} and \ref{lem4}.

\subsubsection{Actions on the coset space of a primitive subgroup}

Now assume that $G=\Alt(m)$ or $G=\Sym(m)$ is primitive in its action on $\Omega$ 
and let $H$ be the stabilizer in $G$ of a point of
$\Omega$. The group $H$ can either be intransitive, imprimitive or
primitive in its action on $\{1,\ldots,m\}$. If $H$ is intransitive,
then the maximality of $H$ in $G$ yields that the action of $G$ on
$\Omega$ is the natural action of $G$ on the $k$-subsets of
$\{1,\ldots,m\}$, for some $k$. Theorem~\ref{actiononsets} dealt with this case. 
If $H$ is imprimitive on $\{1,\ldots,m\}$ then the maximality of $H$ yields that 
the action of $G$ on $\Omega$ is the natural action of $G$ on uniform partitions, 
which was dealt with in Proposition \ref{uniformpartitions}.  
We now consider the remaining cases where  $H$
is primitive. We make use of the following result of 
Mar\'oti~\cite[Theorem~$1.1$]{Maroti}, which improves a result of Saxl and the second author~\cite{PrSa}.

\begin{lemma}\label{orders}
Let $H$ be a primitive permutation group of degree $m$. Then one of
the following holds.
\begin{description}
\item[(i)]$H$ is a subgroup of $\Sym(r)\wr\Sym(s)$ containing
  $\Alt(r)^s$, where the action of $\Sym(r)$ is on $k$-sets
  from $\{1,\ldots,r\}$ and the wreath product has the product action of
  degree $m=\binom{r}{k}^s$;
\item[(ii)]$H=M_{11}$, $M_{12}$, $M_{23}$ or $M_{24}$ in its
 natural $4$-transitive action;
\item[(iii)]$|H|\leq m\cdot \prod_{i=0}^{\lfloor\log_2(m)\rfloor-1}(m-2^i)$.
\end{description}
\end{lemma}

We also need the following refined version of Stirling's formula~\cite{Robbins}.

\begin{lemma}\label{stirling}
For every $n\geq 1$, 
$$
\sqrt{2\pi  n}e^{\frac{1}{12n+1}}\left(\frac{n}{e}\right)^n\leq n!\leq \sqrt{2\pi n}
e^{\frac{1}{12n}}\left(\frac{n}{e}\right)^n.
$$
\end{lemma}

\begin{lemma}\label{yetanother}Let $m$ and $k$ be positive integers, let 
$p$ be a prime and let $0<\alpha<1$. Write $r=m-kp$. If $0\leq r\leq \alpha 
m$, then $p^k(r/e)^r(k/e)^k(m/e)^{-m}\leq (m/e)^{((-1+\alpha)/2)m}$.
\end{lemma}
\begin{proof}
We have
\begin{eqnarray*}
p^k(r/e)^r(k/e)^k(m/e)^{-m}&=&e^{k\log(p)-r-k+m}r^rk^km^{-m}=e^{k\log(p)-k+kp}r^rk^km^{-m}\\
&=&e^{k(p-1+\log(p))} e^{r\log(r)+k\log(k)-m\log(m)}\\
&\leq& e^{k(p-1)} e^{r\log(r)+k\log(m)-m\log(m)}\\
&=&e^{m-k-r} e^{r\log(r)+k\log(m)-m\log(m)}\\
&\leq&e^{m-k-r}e^{r\log(m)+k\log(m)-m\log(m)}\\
&=&e^{(\log(m)-1)(r+k-m)}=(m/e)^{r+k-m}\\
&=&(m/e)^{\frac{1}{p}\left(r+kp\right)+\frac{p-1}{p}r-m}\leq (m/e)^{\frac{1}{p}m+\frac{\alpha(p-1)}{p}m-m}\\
&\leq& (m/e)^{((\alpha-1)/2)m},
\end{eqnarray*}
where in the first inequality we used $\log(k)\leq \log(m/p)=\log(m)-\log(p)$ and 
in the last inequality we used that the function $p\mapsto 1/p+\alpha(p-1)/p$ has 
a maximum at $p=2$, with value $(\alpha+1)/2$.
\end{proof}

\begin{lemma}\label{newlemma1}Let $A$ and $B$ be finite permutation groups 
on $\Delta$ and $\{1,\ldots,\ell\}$ respectively, such that $A$ 
is not regular on $\Delta$, and consider the product action of
$G=A\wr B$ on $\Omega=\Delta^\ell$. Then 
$\max\{\fpr_\Omega(g)\mid {g\in G\setminus\{1\}}\}=\max\{\fpr_\Delta(x)\mid {x\in A\setminus\{1\}}\}$.
\end{lemma}

\begin{proof}Write $m_\Omega:=\max\{\fpr_\Omega(x)\mid {x\in G\setminus\{1\}}\}$ and 
$m_\Delta:=\max\{\fpr_\Delta(x)\mid {x\in A\setminus\{1\}}\}$. 
Observe that $1/|\Delta|\leq m_\Delta$ because $A$ is not regular.
Let $h\in A\setminus\{1\}$ with $\fpr_\Delta(h)=m_\Delta$. 
Then the permutation $g=(h,1,\ldots,1)\in A^\ell\leq G$ and 
$\fpr_\Omega(g)=\fpr_\Delta(h)=m_\Delta$. Thus $m_\Omega\geq 
m_\Delta$. We now prove the reverse inequality.

Let $g=(h_1,\ldots,h_\ell)\in A^\ell$ with $g\neq 1$. Then $\fpr_\Omega(g)=\fpr_\Delta(h_1)\cdots \fpr_\Delta(h_\ell)\leq m_\Delta$. Next, let $g\in G\setminus A^\ell$ with $g=(h_1,\ldots,h_\ell)\sigma$, for some $h_1,\ldots,h_\ell\in A$ and $\sigma\in B\setminus\{1\}$. Relabelling the index set $\{1,\ldots,\ell\}$ if necessary, we may assume that $(1,\ldots,k)$ is a non-identity cycle of $\sigma$. Let $\omega=(\delta_1,\ldots,\delta_\ell)\in \Omega$. Now, 
\[
\omega^g=(\delta_k^{h_k},\delta_1^{h_1},\cdots,\delta_{k-2}^{h_{k-2}},\delta_{k-1}^{h_{k-1}},\delta_{k+1}',\ldots,\delta_{\ell}')
\]
for some $\delta_{k+1}',\ldots,\delta_{\ell}'\in\Delta$. In particular, if $\omega^g=\omega$, then 
$\delta_1=\delta_k^{h_k},\,\delta_2=\delta_1^{h_1},\ldots,\,\delta_{k}=\delta_{k-1}^{h_{k-1}}$, that is,
\[
\delta_k=\delta_1^{(h_k)^{-1}},\,\delta_{k-1}=\delta_1^{(h_{k-1}h_k)^{-1}},\,\ldots,\,\delta_2=\delta_1^{(h_2\cdots h_{k-1}h_k)^{-1}}.
\]
From this we deduce that $k-1$ coordinates of $\omega$ are uniquely determined  by the first coordinate of $\omega$. Since $k\geq 2$, we obtain  $\fpr_\Omega(g)\leq |\Delta|^{\ell-1}/|\Delta|^\ell=1/|\Delta|\leq m_\Delta$. 
\end{proof}

We now deal with primitive actions of $\Sym(m)$ where the stabilizer is primitive on $\{1,\dots,m\}$.

\begin{proposition}\label{thm2}Let $G$ be a primitive group on $\Omega$
  with socle $\Alt(m)$ such that, for $\omega\in \Omega$, the stabilizer $G_\omega$
  is primitive on $\{1,\dots,m\}$. Then either each element of $G$ has a regular 
cycle on $\Omega$, or $m=6$, $G=\Sym(6)$ and $G_\omega=\PGL_2(5)$.
\end{proposition}

\begin{proof}
Write $H:=G_\omega$. We use the trichotomy offered in
Lemma~\ref{orders}, and our first strategy is to apply the criterion in Lemma~\ref{lemma:apeman}.

\medskip\noindent
{\it Case: $m\geq47$ and part~(iii) but not part~(i) of Lemma~\ref{orders} holds for $H$.}\quad  
Let $x\in H$ with $p:=|x|$ prime. By~\cite[Corollary~$1$]{GM}, 
\begin{equation}\label{gur}
\mathrm{fpr}_{\{1,\ldots,m\}}(x)\leq\frac{4}{7}.
\end{equation}
Let $k$ be the number of cycles of $x$ of length $p$ in its action on $
\{1,\ldots,m\}$ and write $r=m-pk$. Then by~\eqref{gur}, 
$r\leq 4m/7$. Write $N_m= m\prod_{i=0}^{\lfloor\log_2(m)\rfloor-1}(m-2^i)$.

For a real number $\gamma\geq 1$,
write $c_\gamma=e^{1/(12\gamma+1)}$ and $C_\gamma=e^{1/(12\gamma)}$. Also set $c_0=C_0=1$. 
Then by Lemmas~\ref{obvious3} and~\ref{stirling},

\begin{eqnarray}\nonumber
\mathrm{fpr}_\Omega(x)&=&\frac{|H\cap x^G|}{|x^G|}\leq
\frac{|H|}{|x^G|}\leq \frac{N_m}{|x^G|}\leq
\frac{N_m}{\frac{m!}{2p^kk!r!}}=\frac{2N_mp^kk!r!}{m!}\\\label{newnew}
&\leq&\frac{2N_mp^{k}\sqrt{2\pi k}\sqrt{2\pi r}C_kC_r}{\sqrt{2\pi
    m}c_m}\left(\frac{k}{e}\right)^{k}\left(\frac{r}{e}\right)^r\left(\frac{m}{e}\right)^{-m}.\\\nonumber
\end{eqnarray}
(Observe that the extra factor of $2$ in the denominator of the third
inequality accounts for the case $G=\Alt(m)$.)
It follows from~\eqref{newnew} and
Lemma~\ref{yetanother} (applied with $\alpha=4/7$) that

\begin{eqnarray}\nonumber
\mathrm{fpr}_\Omega(x)&\leq &\frac{2N_m\sqrt{\pi m}\sqrt{8\pi m/7
  }C_{r}C_{k}}{\sqrt{2\pi m}c_m}
\left(\frac{m}{e}\right)^{\frac{-3m}{14}}\label{eqeq555}
\leq 1.2\cdot \sqrt{16\pi
    m/7}N_m \left(\frac{m}{e}\right)^{\frac{-3m}{14}}.\\
\end{eqnarray}
(Observe that in the last
inequality we have $c_{m}\geq 1$ and, for each $\gamma$, we have $C_\gamma\leq C_1=1.09$). Denote the
right hand side of~\eqref{eqeq555} by $\beta_m$ and observe that this is
a function of $m$ only.

Now let $g\in G$ with $|g|$ square-free. Lemma~\ref{numbertheory} gives an
upper bound for $\omega(|g|)$ depending only on $|g|$ and
then Lemma~\ref{mass} gives an upper bound on $|g|$ depending only on
$m$. Call $\alpha_m$ this function of $m$. It follows, with the help
of a computer, that $\alpha_m\beta_m<1$ for every $m\geq 47$. Thus, 
in these cases, the theorem follows from Lemma~\ref{lemma:apeman}.

\medskip\noindent
{\it Case: $m\leq144$.}\quad  
Here we assume that $m\leq 144$ and deal with every primitive group 
$H$ of degree $m$. In particular this completes our analysis of
groups satisfying parts~(ii) and~(iii) of Lemma~\ref{orders}.
We use a computer. For each
possible $H$  and $G=\Alt(m)$ or $\Sym(m)$ we determine the maximum of
$|H\cap x^G|/|x^G|$ as $x$ runs through the non-identity elements of
$H$ of prime order. Once this number is obtained we multiply it by the maximum
$\omega(|g|)$, as $g$ runs through the elements of $G$. In each case the
product of these two numbers is $<1$ unless $m\leq 12$. Now that $m$
is very small we can afford to construct, for each maximal subgroup
$H$ of
$G=\Alt(m)$ or $G=\Sym(m)$, respectively, such that $H$ is primitive 
on $\{1,\ldots,m\}$, the permutation representation of $G$ on the cosets of
  $H$ and test each element of square-free order. In each case the
  theorem is valid ($G=\Sym(6)$ and $H=\PGL_2(5)$ is the only example
  where there exists an element $g$, of order $6$, not having a cycle
  of length $|g|$).

\medskip\noindent
{\it Case: $m>144$ and part~(i) of Lemma~\ref{orders} holds.}\quad  
Here $\soc(H)=\Alt(c)^\ell$, $m=\binom{c}{d}^\ell$, 
for some $c, d, \ell$ with $1\leq d<c/2$ and $d\ell\geq2$, and the 
action of $H$ on $\{1,\ldots,m\}$ is the natural product
action on the set of $\ell$-tuples from the set $\Delta_d$ of  $d$-sets of 
$\{1,\ldots,c\}$. From Lemma~\ref{newlemma1} (applied with $A:=\Sym(c)$, $G:=
A\wr\Sym(\ell)$ and $\Delta=\Delta_d$), we see that 
$\max\{\fpr_{\{1,\ldots,m\}}(x)\mid x\in H\setminus\{1\}\}\leq
\max\{\fpr_{\Delta_d}(x)\mid x\in \Sym(c),x\neq 1\}$.
It is easy to see
that, for a permutation $x\in \Sym(c)$ with $x\neq 1$, we have
$\mathrm{fpr}_{{\Delta_d}}(x)\leq 1-2/c$
(the maximum is actually achieved with $x$ a transposition, $d=1$ and $\Delta_d=\{1,\ldots,c\}$).   In particular, for $x\in H$ with $x\neq 1$, we have
$$
\mathrm{fpr}_{\{1,\ldots,m\}}(x)\leq 1-\frac{2}{c}
$$
and we proceed exactly as
in~\eqref{eqeq555}, $N_m$ (an upper bound for $|H|$) is replaced by 
$c!^\ell \ell!$ and the constant $4/7$ (an upper bound for 
$\mathrm{fpr}_{\{1,\dots,m\}}(x)$) is replaced by
$1-2/c$. Namely, following the computations in~\eqref{newnew} 
and~\eqref{eqeq555} (and applying Lemma~\ref{yetanother} with $\alpha=1-2/c$) we obtain
(as an analogy to \eqref{eqeq555}) for an element $x\in H$ of prime order $p$ with 
$r=m-pk$ fixed points in $\{1,\dots,m\}$ (so $r\leq (1-\frac{2}{c})m$),
\begin{eqnarray*}
\mathrm{fpr}_\Omega(x)&\leq& \frac{2(c!^\ell \ell!)\sqrt{2\pi \ell}\sqrt{2\pi k
  }C_{r}C_{k}}{\sqrt{2\pi m}c_m}
\left(\frac{m}{e}\right)^{\frac{-m}{c}}
\leq \frac{2(c!^\ell \ell!)\sqrt{2\pi m}\sqrt{2\pi m
  }C_{r}C_{k}}{\sqrt{2\pi m}c_m}
\left(\frac{m}{e}\right)^{\frac{-m}{c}}\\
&=&\frac{2\sqrt{2\pi m}(c!^\ell \ell!)C_{r}C_{k}}{c_m}
\left(\frac{m}{e}\right)^{\frac{-m}{c}}\leq 2.4\sqrt{2\pi m}(c!^\ell \ell!)
\left(\frac{m}{e}\right)^{\frac{-m}{c}},
\end{eqnarray*}
where the last inequality follows because $c_m>1$ and $C_\gamma\leq C_1\leq 1.09$.
In this way we obtain an upper bound on
$\fpr_\Omega(x)$ as a function of $\ell$, $c$ and $d$. It is again a
computation, with the help of a computer, to show that this function
times $\alpha_m$ is always less than $1$, except when $\ell=2$, $d=1$ and $c\leq 12$. 
However, for the corresponding  values of $m$, we have $m\leq
144$, which we assume is not the case here. This completes the proof.
\end{proof}

\section{Concluding remarks}

We finish by bringing together the various threads to prove Theorem \ref{thrm:main}.

\begin{proof}[Proof of Theorem~\ref{thrm:main}]
Let $G\leq\Sym(\Omega)$ be a primitive group that contains an element  with no regular cycle. 
By Theorem \ref{thrm:affine}, $G$ is not of affine type and by Theorem \ref{thrm:SD}, 
$G$ is not of Diagonal type. Thus we may assume that $G\leqslant H\wr \Sym(r)$, where 
either $r=1$ and $G=H, \Omega=\Delta$, or
$r\geq2$ and $G$ preserves a product structure 
$\Omega=\Delta^r$. We can choose $r$ maximal so that $H$ is primitive and does not preserve a 
product structure on $\Delta$. (Otherwise, if $H$ preserves $\Delta=\Gamma^k$ 
with $H\leqslant R\wr \Sym(k)$ then $G$ preserves the structure $\Omega=\Gamma^{kr}$ 
and $G\leqslant R\wr \Sym(kr)$.)  Thus we have $G\leqslant H\wr \Sym(r)$, where 
$H$ is primitive of almost simple or of affine or of diagonal type. Since $G$ contains 
an element $g$ with no regular cycles on $\Omega$, Theorem \ref{thrm:PA} implies that 
some element of $H$ has no regular cycle on $\Delta$. It then follows 
from Theorems~\ref{thrm:affine} and~\ref{thrm:SD} that 
$H$ is an almost simple group.  Let $T=\soc(H)$. Then by \cite{LPS1}, $\soc(G)=T^r$. 
Thus we have $T^r\vartriangleleft G\leq H\wr\Sym(r)$ in product action on $\Omega=\Delta^r$,  for some $r\geq 1$,
for a primitive almost simple group $H\leq\Sym(\Delta)$ with socle $T$, such that 
some element of $H$ has no regular cycle.
By Theorems \ref{thrm:sporadic} and \ref{thrm:excep}, $T$ is neither a sporadic simple group
nor an exceptional group of Lie type. If $T$ is a classical simple group then
the conclusion of Theorem~\ref{thrm:main} holds, so we may assume that $T=\Alt(m)$ for some $m\geq 7$.
Thus $H=\Alt(m)$ or $\Sym(m)$.   
By assumption, $(H,\Delta)$ is not the $k$-set action of $\Alt(m)$ or $\Sym(m)$, for any $k$. 
Thus a stabilizer $H_\delta$ (for $\delta\in\Delta$) is transitive on $\{1,\dots,m\}$.
By Proposition~\ref{uniformpartitions}, $H_\delta$ is primitive on $\{1,\dots,m\}$ (since $m\geq5$), and by
Proposition~\ref{thm2}, the only possibility for $H$ is $H=\Sym(6)$, but we have $m\geq 7$. 
This contradiction completes the proof. 
\end{proof}


\thebibliography{10}
\bibitem{onlineAtlas}Online Atlas of Finite Group Representations, \href{http://web.mat.bham.ac.uk/atlas/v2.0/}{http://web.mat.bham.ac.uk/atlas/v2.0/}

\bibitem{AIPS}A.~Azad, M.~A.~Iranmanesh, C.~E.~Praeger, P.~Spiga,
  Abelian coverings of finite general linear groups and an application
  to their non-commuting graphs,
  \textit{J.~Algebr.~Comb. }\textbf{34} (2011),  683--711.

\bibitem{BS}E.~Bach, J.~Shallit, Section~$2.7$ in Algorithmic Number
  Theory, Vol. 1: Efficient Algorithms. Cambridge, MA: MIT Press,
  1996.

\bibitem{magma}W.~Bosma, J.~Cannon, C.~Playoust, The Magma algebra system. I. The user language, \textit{J. Symbolic Comput.} \textbf{24} (1997), 235--265.


\bibitem{cameron}
P.~J.~Cameron,   Permutation groups, \textit{London Math. Soc. Student Texts} \textbf{45}, Cambridge University Press, Cambridge. 1999.


\bibitem{ATLAS}J.~H.~Conway, R.~T.~Curtis, S.~P.~Norton, R.~A.~Parker,
  R.~A.~Wilson, \textit{Atlas of Finite Groups}, Clarendon Press,
  Oxford, 1985.

\bibitem{DM}J.~D.~Dixon, B.~Mortimer, \textit{Permutation groups}, Graduate Texts in Mathematics \textbf{163}, Springer-Verlag, New York, 1996.

\bibitem{EZal}
L.~Emmett, A.~E.~Zalesski, On regular orbits of elements of classical groups in their 
permutation representations. \textit{Comm. Algebra} \textbf{39} (2011),  3356--3409. 


\bibitem{GMPSorders}S.~Guest, J.~Morris, C.~E.~Praeger, P.~Spiga, On the maximum orders of elements 
of finite almost simple groups and primitive permutation groups, \textit{Trans. Amer. Math. Soc.}, to appear.

\bibitem{GS}
S. Guest, P. Spiga, {Finite primitive groups and regular orbits of group elements, in preparation.

\bibitem{GM}R.~Guralnick, K.~Magaard, On the minimal degree of a
  primitive permutation group, \textit{J. Algebra} \textbf{207}
  (1998), 127--145.

\bibitem{HH}
B.~Hartley, T.~O.~Hawkes, \emph{Rings, modules and linear algebra.}
Chapman and Hall, London,  1970. 

\bibitem{KL}P.~Kleidman, M.~Liebeck, The Subgroup Structure of the
  Finite Classical Groups, \textit{London Mathematical Society Lecture
  Note Series} \textbf{129}, Cambridge University Press, Cambridge, 1990.


\bibitem{LLS}R.~Lawther, M.~W.~Liebeck, G.~M.~Seitz, Fixed point
  ratios in actions of finite exceptional groups of Lie type,
  \textit{Pacific Journal of Mathematics} \textbf{205} (2002), 393--464.



\bibitem{LPS2}M.~W.~Liebeck, C.~E.~Praeger, J.~Saxl, A Classification
  of the maximal subgroups of the finite alternating and symmetric
  groups, \textit{J.~Algebra} \textbf{111} (1987), 365--383.


\bibitem{LPS1}M.~W.~Liebeck, C.~E.~Praeger, J.~Saxl, On the O'Nan-Scott
  theorem for finite primitive permutation
  groups, \textit{J. Austral. Math. Soc. Ser. A} \textbf{44} (1988),
  389-–396.

\bibitem{LS}M.~W.~Liebeck, J.~Saxl, Minimal degrees of primitive
  permutation groups, with an application to monodromy groups of
  covers of Riemann surfaces, \textit{Proc. London Math. Soc.~(3)}
  \textbf{63} (1991), 266--314.

\bibitem{Maroti}A.~Mar\'oti, On the orders of primitive groups,
  \textit{J. Algebra} \textbf{258} (2002), 631--640.

\bibitem{Mass}J.~P.~Massias, J.~L.~Nicolas, G.~Robin, Effective Bounds
  for the Maximal Order of an Element in
the Symmetric Group, \textit{Mathematics of Computation} \textbf{53} (1989), 665--678.

\bibitem{25}C.~E.~Praeger, Finite quasiprimitive graphs, in Surveys in
  combinatorics, \textit{London Mathematical Society Lecture Note
    Series}, vol. 24 (1997), 65--85.

\bibitem{PrSa}C.~Praeger, J.~Saxl, On the order of primitive permutation groups, \textit{Bull. London Math. Soc.} \textbf{12} (1980), 303--308. 

\bibitem{Potter}W.~Potter,
Nonsolvable groups with an automorphism inverting many elements,
\textit{Arch. Math.} \textbf{50} (1988), 292--299.

\bibitem{Robbins}H.~Robbins, A remark on Stirling's formula,
  \textit{Amer. Math. Montly} \textbf{62} (1955), 26--29.

\bibitem{Robin}G.~Robin,
Estimation de la fonction de Tchebychef $\Theta$ sur le $k$-i\`eme
nombre premier et grandes valeurs de la fonction $\omega(n)$ nombre de
diviseurs premiers de $n$,  \textit{Acta Arith.} \textbf{42} (1983), 367--389.

\bibitem{Rudin}D.~J.~Rusin, What is the probability that two elements
  of a finite group commute?, \textit{Pacific J.~Math.} \textbf{82}
  (1979), 237--247.


\bibitem{SZ2}
J.~Siemons, A.~Zalesskii, Intersections of matrix algebras and permutation representations of PSL(n,q).
\textit{J. Algebra}  \textbf{226} (2000), 451--478.

\bibitem{SZ1}
J.~Siemons, A.~Zalesski\u{i}, Regular orbits of cyclic subgroups in permutation 
representations of certain simple groups, \emph{J. Algebra} \textbf{256} (2002), 611--625.
\end{document}